\documentclass[twoside,a4paper,10pt,reqno,intlimits,envcountsame]{svjour3}
\usepackage{amsmath}
\usepackage{color}
\usepackage{amssymb}
\usepackage{amsfonts}
\usepackage{amsxtra}
\usepackage{xspace}
\usepackage{multirow}

\usepackage{enumitem}
\setenumerate{label={\rm (\alph{*})}}

\usepackage[rose,frak,hyperref]{paper_diening}

\smartqed

\allowdisplaybreaks

\usepackage{graphicx}
\hypersetup{%
  pdftitle={},
  pdfsubject={},
  pdfauthor={},
  pdfkeywords={},
  pdfcreator={pdfLaTeX},
  pdfproducer={pdfLaTeX},
  }

\newtheorem{assumption}[theorem]{Assumption}

\numberwithin{equation}{section}
\numberwithin{theorem}{section}

\newcommand{\meantmp}[2]{#1\langle{#2}#1\rangle}
\newcommand{\mean}[1]{\meantmp{}{#1}}

\newcommand{\Pidiv}{\ensuremath{\Pi^{\divergence}_h}}

\newcommand{\Rn}{{\setR^n}}

\newcommand{\PP}{\mathcal{P}}
\newcommand{\PPln}{\mathcal{P}^{{\log}}}

\providecommand{\bfST}{{\bfS_{\mathcal{T}}}}
\providecommand{\bfFT}{{\bfF_{\mathcal{T}}}}

\newcommand{\problemP}{\text{\rm \bf(P)}\xspace}
\newcommand{\problemQ}{\text{\rm \bf (Q)}\xspace}
\newcommand{\problemPh}{\text{\rm \bf (P$_h$)}\xspace}
\newcommand{\problemQh}{\text{\rm \bf (Q$_h$)}\xspace}
\newcommand{\px}{{p(\cdot)}}
\newcommand{\PiY}{\Pi^Y_h}

\renewcommand{\frq}{{\boldsymbol {\mathfrak q}}}
\renewcommand{\frp}{{\boldsymbol {\mathfrak p}}}

\journalname{Numerische Mathematik}

\begin{document}

\title{Convergence Analysis for a Finite Element Approximation of a
  Steady Model for Electrorheological Fluids}
\titlerunning{Finite Element for Electrorheological fluids}
\author{Luigi C. Berselli \and Dominic Breit \and Lars Diening}

\institute{Luigi C. Berselli \at Dipartimento di Matematica,
  Universit{\`a} di Pisa, Via F.~Buonarroti 1/c, I-56127 Pisa, ITALY.
  \\
  \email{berselli@dma.unipi.it} 
  \and 
  Dominic Breit \and Lars Diening \at Institute of Mathematics, LMU
  Munich Theresienstr., 39 D-80333 Munich, Germany
  \\
  \email{diening@math.lmu.de}, \email{breit@math.lmu.de}
}

\maketitle


\begin{abstract}
  In this paper we study the finite element approximation of systems
  of $\px$-Stokes type, where $\px$ is a (non constant) given function
  of the space variables. We derive --in some cases optimal-- error
  estimates for finite element approximation of the velocity and of
  the pressure, in a suitable functional setting.
  \\[3mm]
\end{abstract}


\maketitle


\section{Introduction}
\label{sec:intro}
The stationary flow of an incompressible homogeneous fluid in a
bounded domain $\Omega\subset\mathbb R^n$ is described by the set of
equations
\begin{align}
    \label{eq:ERF}
    - \divergence \mathcal{S}+\divergence( \bfv\otimes\bfv) + \nabla q
    &=\bff\,, \qquad \divergence \bfv = 0\quad \text{in}\quad \Omega.
\end{align}
Here $\bfv:\Omega\rightarrow\mathbb R^n$ and
$q:\Omega\rightarrow\mathbb R$ are the unknown velocity field and
pressure respectively, whereas $\bff:\Omega\rightarrow\mathbb R^n$ is
a given volume force. A popular model for Non-Newtonian (Newtonian
if $p=2$) fluids is the power-law model
\begin{align} 
  \label{eq:n3.3}
  \mathcal{S}=\mathcal{S}(\bfD\bfv)=\mu(\kappa+|\bfD\bfv|)^{p-2}\bfD\bfv,
\end{align}
with $\mu>0$, $\kappa\in[0,1]$, and $1 < p < \infty$. The extra stress
tensor $\mathcal{S}(\bfD \bfv)$ depends on $\bfD\bv:=\tfrac
12(\nabla\bv+\nabla\bv^\top)$, the symmetric part of the velocity
gradient $\nabla \bv$.  Physical interpretation and discussion of some
non-Newtonian fluid models can be found, e.g., in~\cite{bird,MRR1995}

In this paper we consider a further generalization of~\eqref{eq:n3.3},
which is motivated by a model introduced
in~\cite{0890.76007,rajagopal01:_mathem_model_elect_mater} to describe
motions of electrorheological fluids, further studied
in~\cite{MR1810360}. Electrorheological fluids are special smart
fluids, which change their material properties due to the application
of an electric field; especially the viscosity can locally change by a
factor of $10^3$ in 1ms. Electrorheological fluids can be used in the
construction of clutches and shock absorbers.  In the model introduced
in~\cite{rajagopal01:_mathem_model_elect_mater} the exponent $p$ is
not a fixed constant, but a function of the electric field $\bfE$, in
particular $p:=p(\abs{\bfE}^2)$.  The electric field itself is a
solution to the quasi--static Maxwell equations and is not influenced
by the motion of the fluid.  In a first preliminary step, it is then
justified to separate the Maxwell equation from~\eqref{eq:ERF} and to
study, for a given function $p : \Omega \to (1, \infty)$, the
system~\eqref{eq:ERF} with $\bfS : \Omega \times \setR_{\sym}^{n\times
  n} \to \setR_{\sym}^{n\times n}$ satisfying for all $x \in \Omega$
and for all $\bfeta \in \setR_{\sym}^{n\times n}$
\begin{align}
  \label{eq:1.3}
  \bfS(x,\bfeta)=\mu(\kappa+|\bfeta|)^{p(x)-2}\bfeta.
\end{align}
This model comprises all the mathematical difficulties of the full
system for electrorheological fluids (as in~\cite{MR1810360}) and the
results below can be directly extended to the general case. In this
first study we consider the case of a slow flow and therefore neglect
the convective term $\divergence(\bfv\otimes\bfv)=(\nabla\bfv)\bfv$.
A reintroduction of this term causes the usual difficulties as for
instance the possible non-uniqueness of the solution. This problem
also arises for the continuous problem. For small data and large
exponents~$p$ one can recover uniqueness. In this situation it should
be possible to generalize the results of this paper to the presence of
the convection term. However, a numerical analysis for small exponents
will be complicated, but this difficulty also appears for constant
exponents.

Therefore, as a first step (to focus on peculiar difficulties of
variable exponents) in this paper we study the numerical approximation 
of steady systems of the $\px$-Stokes type
%
\begin{equation}
  \label{eq:linear_pfluid}
  \begin{aligned}
    -\divo \bfS(\cdot,\bfD\bv)+\nabla q&=\bff\qquad&&\text{in
    }\Omega,
    \\
    -\divo\bv&=0\qquad&&\text{in }\Omega,
    \\
    \bfv &= \bfzero &&\text{on } \partial \Omega,
  \end{aligned}
\end{equation}
with $\bfS$ with variable exponent given by~\eqref{eq:1.3}.  Our
approach is based on conforming finite element spaces satisfying the
classical discrete inf-sup condition. We assume that
$\Omega\subset\setR^n$ is a polyhedral, bounded domain.

The mathematical investigation of fluids with shear-dependent
viscosities ($p=$const.) started with the celebrated works of
O.A.~Ladyzhenskaya~\cite{lady-bo} and J.-L.~Lions~\cite{Lio1969} in
the late sixties. In recent years there has been an enormous progress
in the understanding of this problem and we refer the reader
to~\cite{FMS,mnrr,Ruz2013} and the references therein for a detailed
discussion.

The first results regarding the numerical analysis date back
to~\cite{San1993}, with improvements in~\cite{BL1994b}) where the
error estimates are presented in the setting of quasi-norms. The
notion of quasi-norm is the natural one for this type of problem,
cf.~\cite{bb,BL1993a,
BL1994b,DieR07,LB1996}, since the
quasi-norm is equivalent to the distance naturally defined by the
monotone operator $-\divergence\mathcal{S}(\bfD \bfv)$. Given the
discrete solution $\bfv_h$ and the continuous solution $\bfv$, the
error is measured as the $L^2$-difference of $\mathcal{F}(\bfD\bfv)$
and $\mathcal{F}(\bfD\bfv_h)$, where
$\mathcal{F}(\bfeta)=(\kappa+|\bfeta|)^{\frac{p-2}{2}}\bfeta$.  For the
system~\eqref{eq:ERF}--\eqref{eq:n3.3}, (without convective term) the following error
estimates are shown for constant $p>1$ in~\cite{BeBeDiRu}:
  \begin{align}
    \label{eq:appr_h1a0}
    \norm{\mathcal{F}(\bfD \bfv)-\mathcal{F}(\bfD
      \bfv_h)}_2&\leq c\, h^{\min \set{1,\frac {p'}2}},
    \\
    \label{eq:appr_h1b0}
    \|q-q_h\|_{p'} &\leq c\,h^{\min\big\{\frac{p'}{2},\frac{2}{p'}\big\}}.
  \end{align}
  For the validity of the above estimates the natural assumptions that
  $\mathcal{F}(\bfD \bfv) \in (W^{1,2}(\Omega))^{n\times n}$ and also
  that $q \in W^{1,p'}(\Omega)$ are made. The convergences rates
  in~\eqref{eq:appr_h1a0} and~\eqref{eq:appr_h1b0} are the best known
  ones.

  The purpose of the present paper is to extend the
  estimates~\eqref{eq:appr_h1a0}-\eqref{eq:appr_h1b0} to the setting
  of variable exponents for the $p(\cdot)$-Stokes system.  Since the
  development of the model for electrorheological fluids
  in~\cite{0890.76007,rajagopal01:_mathem_model_elect_mater,Ruz2004}
  there has been a huge progress regarding its mathematical
  analysis~\cite{dms,AcMi,BF,DER} and especially the precise
  characterization of the corresponding functional
  setting~\cite{DiHaHaRu}. However there are only very few results
  about the numerical analysis, see for instance results for the
  time-discretization in~\cite{Die2002}.

  We observe that in~\cite{ECP} a time-dependent system with the same
  stress-tensor~\eqref{eq:1.3} and (smoothed) convective terms is
  studied and the convergence of the finite element approximation is
  shown without convergence rate. To our knowledge no quantitative
  estimate on the convergence rate for problems with variable
  exponents is known, while recent results for the $\px$-Laplacian
  (i.e. the scalar system without pressure) can be found
  in~\cite{BrDiSc}

  The main purpose of the present paper is to obtain precise
  convergence estimates for the system~\eqref{eq:linear_pfluid},
  generalizing to the variable exponent the
  estimates~\eqref{eq:appr_h1a0}-\eqref{eq:appr_h1b0}. By assuming
  H\"older regularity on the exponent $\px$, the main results we will
  prove (see Theorem~\ref{thm:appr_h1}) are the following estimates
  \begin{align*}
    \norm{\bfFT(\cdot,\bfD \bfv)-\bfFT(\cdot,\bfD \bfv_h)}_2&\leq c\,
    \big(h^{\min \set{1,\frac {(p^+)'}2}}+h^\alpha\big),
    \\
    \|q-q_h\|_{p'(\cdot)} &\leq
    c\,\Big(h^{\frac{\min\set{((p^+)')^2,4}}{2(p^{-})'}}+h^{\alpha}\Big).
  \end{align*}
  Here $\bfFT$ is a locally constant approximation to
  $\bfF(x,\bfeta)=(\kappa+|\bfeta|)^{\frac{p(x)-2}{2}}\bfeta$,
  see~\eqref{eq:def:F_Tau}, the number $\alpha\in(0,1]$ is the
  H\"older exponent of $\px$ and $p^+$ and $p^{-}$ supremum and
  infimum value of $\px$, respectively. As usual $(p^+)'$ and $(p^-)'$
  are their conjugate exponents. Our analysis is
  based on the recent studies in~\cite{BrDiSc} about the finite
  element approximation of the $\px$-Laplacian and on the numerical
  analysis in~\cite{BeBeDiRu} for the $p$-Stokes system.
  \\
  \\
  \textbf{Plan of the paper:} In Sec.~\ref{sec:prel} we recall the
  basic results on variable exponent space, we will use. In
  Sec.~\ref{sec:problem} we recall the basic existence results for the
  $p(\cdot)$-Stokes system, the finite element setting, and we state
  the main results of the paper.  The convergence analysis of the
  velocity is presented in Sec.~\ref{sec:velocity} and the convergence
  analysis of the pressure in Sec.~\ref{sec:appr-with-pressure}. An
  appendix with some technical results on Orlicz spaces is also added
\section{Variable exponent spaces}
\label{sec:prel}
\noindent
For a measurable set $E \subset \Rn$ let $\abs{E}$ be the Lebesgue
measure of~$E$ and $\chi_E$ its characteristic function. For $0 <
\abs{E} < \infty$ and $f \in L^1(E)$ we define the mean value of $f$
over $E$ by
\begin{align*}
   \mean{f}_E:=\dashint_{E}f \,dx:=\frac{1}{\abs{E}}\int_E f \,dx.
\end{align*}
For an open set $\Omega \subset \Rn$ let $L^0(\Omega)$ denote the set
of measurable functions.

Let us introduce the spaces of variable exponents~$L^\px$.  We use the
same notation used in the recent book~\cite{DiHaHaRu}.  We define
$\PP$ to consist of all $p \in L^0(\Rn)$ with $p\,:\, \Rn \to
[1,\infty]$ (called variable exponents).  For $p \in\PP$ we define
$p^-_\Omega := \essinf_\Omega p$ and $p^+_\Omega := \esssup_\Omega p$.
Moreover, let $p^+ := p^+_{\Rn}$ and $p^- := p^-_{\Rn}$.

For $p \in \PP$ the generalized Lebesgue space $L^{\px}(\Omega)$ is
defined as
\begin{align*}
  L^\px(\Omega):=\biggset{f \in L^0(\Omega) \,:\,
    \norm{f}_{L^\px(\Omega)} < \infty},
\end{align*}
where
\begin{align*}
  \norm{f}_{\px}:=\norm{f}_{L^\px(\Omega)}:=
  \inf\biggset{\lambda>0\,:\,\int_{\Rn} \biggabs{
      \frac{f(x)}{\lambda}}^{p(x)} \,dx \leq 1}.
\end{align*}
By using a standard notation, by  $\|\cdot\|_p$ we mean
the usual Lebesgue 
$L^p$-norm, for a fixed $p\geq1$. %
%
We say that a function $g \colon \Rn \to \setR$ is {\em
  $\log$-H{\"o}lder continuous} on $\Omega$ if there exist constants
$c \geq 0$ and $g _\infty \in \setR$ such that
\begin{align*}
  \abs{g (x)-g (y)} &\leq \frac{c}{\log
    (e+1/\abs{x-y})} &&\text{and}& \abs{g (x) - g _\infty}
  &\leq \frac{c}{\log(e + \abs{x})},
\end{align*}
for all $x\not=y\in \Rn$. The first condition describes the so called
local $\log$-H{\"o}lder continuity and the second the decay condition.
The smallest such constant~$c$ is the $\log$-H{\"o}lder constant
of~$g$. We define $\PPln$ to consist of those exponents $p\in
\PP$ for which $\frac{1}{p} \,:\, \Rn \to [0,1]$ is $\log$-H{\"o}lder
continuous. By $p_\infty$ we denote the limit of~$p$ at infinity,
which exists for $p \in \PPln$.  If $p \in \PP$ is bounded, then $p
\in \PPln$ is equivalent to the $\log$-H{\"o}lder continuity of $p$.
However, working with $\frac 1p$ gives better control of the constants
especially in the context of averages and maximal functions.
Therefore, we define $c_{\log}(p)$ as the $\log$-H{\"o}lder constant
of~$1/p$. Expressed in~$p$ we have for all $x,y \in \Rn$
\begin{align*}
  \abs{p(x) - p(y)} \leq \frac{(p^+)^2 c_{\log}(p)}{\log
    (e+1/\abs{x-y})} \qquad \text{and }\qquad \abs{p(x) - p_\infty}
  \leq \frac{(p^+)^2 c_{\log}(p)}{\log(e + \abs{x})}.
\end{align*}
For a cube $Q \subset \setR^n$ we denote by $\ell(Q)$
its side length and we have the following results.
\begin{lemma}[Lemma 2.1 in~\cite{BrDiSc}]
  \label{lem:pxpy} 
  Let $p \in \PPln(\Rn)$ with $p^+<\infty$ and $m>0$. Then for every
  cube $Q \subset \Rn$ with $\ell(Q) \leq 1$, $\kappa\in [0,1]$, and $t\geq
  0$ such that $\abs{Q}^m \leq t \leq \abs{Q}^{-m}$, then 
  \begin{align*}
    (\kappa + t)^{p(x)-p(y)} \leq c,
  \end{align*}
  for all $x,y \in Q$. The constant depends on $c_{\log}(p)\,,m$, and
  $p^+$.
\end{lemma}
For every convex function~$\psi$ and every cube~$Q$ we have by
Jensen's inequality
\begin{align}
  \label{eq:jensen}
  \psi\bigg(\dashint_Q \abs{f(y)}\,dy\bigg) &\leq \dashint_Q \psi(
  \abs{f(y)})\,dy.
\end{align}
This simple but crucial estimate allows for example to transfer the
$L^1$-$L^\infty$ estimates for the interpolation operators to the
setting of Orlicz spaces, see~\cite{DieR07}. A suitable analogue for
variable exponent spaces bounds $ ( \dashint_Q \abs{f(y)} \,dy
)^{p(x)}$ in terms of $\dashint_Q \abs{f(x)}^{p(x)} \,dx$ (but an additional
error term appears). In order to quantify this let us introduce the
notation
\begin{align*}
  \phi(x,t) &:= t^{p(x)},\quad 
  (M_Q \phi)(t) := \dashint_Q \phi(x,t)\,dx,\quad
  M_Q f := \dashint_Q \abs{f(x)}\,dx.
\end{align*}
For our finite element analysis we need this estimate extended to the
case of shifted Orlicz functions. For constant~$p$ this has been done
in~\cite{DieR07}. We define the shifted functions $\phi_a$ for $a\geq
0$ by
\begin{align*}
  \phi_a(x,t) &:= \int_0^t \frac{\phi'(x,a+\tau)}{a+\tau} \tau\,d\tau,
\end{align*}
 where the prime denotes the partial derivative of
$\phi(x,t)$ with respect to the variable $t$.

Then $\phi_a(x,\cdot)$ is the shifted N-function of $t \mapsto
t^{p(x)}$, see~\eqref{eq:def_shift}. Note that characteristics and
$\Delta_2$-constants of $\phi_a(x,\cdot)$ are uniformly bounded with
respect to~$a\geq 0$ if $1 < p^- \leq p^+ < \infty$, see
Section~\ref{sec:Orlicz spaces}.

We recall three fundamental results we will use in the sequel.
\begin{theorem}[Shifted key estimate, Thm. 2.5 in~\cite{BrDiSc}]
  \label{thm:jensenpxshift}
  Let $p \in \PPln(\Rn)$ with $p^+<\infty$. Then for every $m>0$ there
  exists $c_1>0$ only depending on $m$, $c_{\log}(p)$, and $p^+$
  such that
  \begin{align*}
    \phi_a(x,M_Q f)&\leq c\, M_Q(\phi_a(|f|)) + c\, \abs{Q}^{m},
  \end{align*}
  for every cube (or ball) $Q \subset \Rn$ with $\ell(Q) \leq 1$, all
  $x \in Q$ and
  all $f\in L^1(Q)$ with
  \begin{align*}
    a+\dashint_Q \abs{f}\,dy \leq \max \set{1, \abs{Q}^{-m}} =
    \abs{Q}^{-m}. 
  \end{align*}
\end{theorem}
\begin{theorem}[Shifted Poincar\'{e} inequality, Thm. 2.4 in~\cite{BrDiSc}]
  \label{thm:poincareshift}
  Let $p \in \PPln(\Rn)$ with $p^+<\infty$. Then for every $m>0$ there
  exists $c>0$ only depending on $m$,  $c_{\log}(p)$, and $p^+$
  such that
  \begin{align*}
    \int_Q \phi_a\bigg(x,
    \frac{\abs{u(x)-\mean{u}_Q}}{\ell(Q)} \bigg)\,dx&\leq c\,
    \int_Q \phi_a(x,\abs{\nabla u(x)})\,dx + c\, \abs{Q}^{m},
  \end{align*}
  for every cube (or ball) $Q \subset \Rn$ with $\ell(Q) \leq 1$ and for all
  all $u\in W^{1,\px}(Q)$ with
  \begin{align*}
    a+\dashint_Q \abs{\nabla u}\,dy \leq \max \set{1, \abs{Q}^{-m}} =
    \abs{Q}^{-m}. 
  \end{align*}
\end{theorem}

\begin{theorem}[shifted Korn inequality]
  \label{thm:kornshift}
  Let $p \in \PPln(\Rn)$ with $p^+<\infty$. Then for every $m>0$ there
  exists $c>0$ only depending on $m$, $c_{\log}(p)$, and $p^+$
  \begin{align*}
    \int_Q \phi_a\bigg(x, \frac{\abs{\bfu(x)-\mathcal R_Q\bfu
        (x)}}{\ell(Q)} \bigg)\,dx&\leq c\, \int_Q \phi_a(x,\abs{\bfD
      \bfu(x)})\,dx + c\, \abs{Q}^{m},
  \end{align*}
  for every cube (or ball) $Q \subset \Rn$ with $\ell(Q) \leq 1$ and
  all $\bfu\in (W^{1,\px}(Q))^n$ with
  \begin{align*}
    a+\dashint_Q \abs{\bfD \bfu}\,dy \leq \max \set{1, \abs{Q}^{-m}}
    = \abs{Q}^{-m}.
  \end{align*}
  Here $\mathcal R_Q$ is a suitable rigid motion, i.e. $\mathcal
  R_Q x=\bfA x+b$ is affine linear with $\bfA^T+\bfA=0$.
\end{theorem}
\begin{proof}
  Due to (2.33)-(2.39) in~\cite{Re} there is a rigid motion $\mathcal
  R_Q\bfu$ such that the difference of $\bfu$ and $\mathcal R_Q\bfu$ can be
  represented as a Riesz-potential of $\bfD\bfu$, i.e. there holds
  \begin{align*}
    |\bfu(x) - \mathcal R_Q\bfu(x)| \leq c \int_Q \frac{|\bfD
      \bfu(y)|}{|x-y|^{n-1}} dy .
  \end{align*}
  Due to this inequality we can prove the claim by the lines of~\cite{BrDiSc}, Thm. 2.4, replacing $\nabla u$ by $\bfD\bfu$.
  \qed
\end{proof}

\section{\texorpdfstring{The $p(\cdot)$-Stokes problem: notation and
    main results}{The p-Stokes problem: notation and main results}}
\label{sec:problem} 
In this section we introduce the main existence results for the
$p(\cdot)$-Stokes and we describe the Finite Element formulation we
will study
%

\subsection{\texorpdfstring{The $p(\cdot)$-Stokes problem} {The p-Stokes
    problem}}
\label{sec:pstokes}
Let us briefly recall some well-known facts about the $p(\cdot)$-Stokes
system~\eqref{eq:linear_pfluid}. Let $\Omega \subset \setR^n$ be a
bounded, polyhedral domain. Then we define the function spaces
\begin{alignat*}{2}
  X &:= \big(W^{1,p(\cdot)}(\Omega)\big)^n, \qquad & V &:=
  \big(W^{1,p(\cdot)}_0(\Omega)\big)^n,
  \\
  Y &:= L^{p'(\cdot)}(\Omega)\,,\qquad & Q &:=
  L^{p'(\cdot)}_0(\Omega):=\biggset{f\in L^{p'(\cdot)}(\Omega)\,:\, \dashint_\Omega
    f\,dx=0 }.
\end{alignat*}
With this notation the weak formulation of
problem~\eqref{eq:linear_pfluid} is the following.
%
\paragraph{\bf Problem \problemQ} For $\bff \in V^* $ find $(\bfv,q) \in V\times Q$ such that
\begin{align*}
      \skp{\bfS(\cdot,\bfD \bfv)}{\bfD \bfxi} - \skp{\divergence \bfxi}{q} &=
    \skp{\bff}{\bfxi} &\forall\, \bfxi \in V,
    \\
    \skp{\divergence \bfv}{\eta} &= 0 &\forall\,
    \eta \in Y.
  \end{align*}

Alternatively, we can reformulate the problem ``hiding'' the pressure:
\paragraph{\bf Problem \problemP}
For $\bff \in L^{p'(\cdot)}(\Omega)$  find $\bfv \in V_{\divergence}$ such that
\begin{align*}
  \skp{\bfS(\cdot,\bfD \bfv)}{\bfD \bfxi} = \skp{\bff}{\bfxi}
  \qquad\forall\, \bfxi \in V_{\divergence},
\end{align*}
where
\begin{align*}
  V_{\divergence} &:= \set{ \bfw \in V \,:\, \skp{\divergence \bfw}{\eta} =
    0 \quad\forall\,\eta \in Y}.
\end{align*}
The names ``Problem~\problemQ'' and ``Problem~\problemP'' are
traditional, see~\cite{BF1991,GR1979}. By using the $\inf$-$\sup$
condition or again the solvability of the divergence equation one easily
checks that the two formulations are equivalent.

The problems~\problemQ and~\problemP have a discrete counterpart,
whose analysis is the ultimate goal of this section. Let $\mathcal{T}$
be a triangulation of our domain~$\Omega$ consisting of
$n$-dimensional simplices. For a simplex $K \in \mathcal{T}$ let $h_K$
denote its diameter and let $\rho_K$ be the supremum of the diameters
of inscribed balls. We assume that $\mathcal T$ is non-degenerate,
i.e., $\max _{K \in \mathcal{T}} \frac {h_k}{\rho_K}\le \gamma_0$.
The global mesh size $h$ is defined by $h := \max _{K \in \mathcal
  T}h_K$.  Let $S_K$ denote the neighborhood of~$K$, i.e., $S_K$ is
the union of all simplices of~$\mathcal{T}$ touching~$K$. One easily
sees that under these assumptions we get that $\abs{K} \sim \abs
{S_K}$ and that the number of simplices in $S_K$ is uniformly bounded
with respect to $K \in \mathcal{T}$.

We denote by $\frP_m(\mathcal{T})$, with $m \in \setN_0$, the space of
scalar or vector-valued continuous functions, which are polynomials of
degree at most $m$ on each simplex $K \in \mathcal{T}$.  Given a
triangulation of $\Omega$ with the above properties and given $k,m \in
\setN_0$ we denote by $X_h \subset (\frP_m(\mathcal{T}))^n$ and
$Y_h\subset \frP_k(\mathcal{T})$ appropriate conforming finite element
spaces defined on $\mathcal{T}$, i.e., $X_h$, $Y_h$ satisfy $X_h
\subset X$ and $Y_h \subset Y$. Moreover, we set $V_h := X_h \cap V$
and $Q_h:= Y_h \cap Q$. For the applications it is convenient to
replace the exponent $p(\cdot)$ by some local approximation.
\begin{align*}
  p_{\mathcal{T}} &:= \sum_{K \in \mathcal{T}}  p(x_K)\chi_K = \sum_{K
    \in \mathcal{T}}  p^-_K \chi_K,
\end{align*}
 where $x_K:=\mathrm{argessinf}_{K} p(x)$, i.e. $p(x_K)= p_K^-$, and consider
\begin{align*}
  \bfS_{\mathcal T}(x,\bfxi)=\sum_{K\in\mathcal T}\chi_K(x)
  \bfS\big(x_K,\bfxi\big),
\end{align*}
instead of $\bfS$.  
Now the discrete counterpart of~\problemP and
\problemQ can be written as follows:
\paragraph{\bf Problem \problemQh} For $\bff \in L^{p'(\cdot)}(\Omega)$
find $(\bfv_h,q_h) \in V_h\times Q_h$ such that
\begin{align}
  \label{eq:weakproblem_h}
  \begin{alignedat}{2}
    \skp{\bfS_{\mathcal T}(\cdot,\bfD \bfv_h)}{\bfD \bfxi_h} - \skp{\divergence \bfxi_h}{q_h} &=
    \skp{\bff}{\bfxi_h} &\qquad&\forall\, \bfxi_h \in V_h,
    \\
    -\skp{\divergence \bfv_h}{\eta_h} &= 0 &\qquad&\forall\,
    \eta_h \in Q_h.
  \end{alignedat}
\end{align}
If $(\bfv_h,q_h) \in V_h\times Q_h$ is  a solution of the ``Problem $(Q_h)$''
then~\eqref{eq:weakproblem_h}$_2$ is satisfied for all $\eta_h \in Y_h$,
since $\divergence \bv_h $ is orthogonal to constants.
\paragraph{\bf Problem \problemPh} For $\bff \in L^{p'(\cdot)}(\Omega)$
 find
$\bfv_h \in V_{h,\divergence}$ such that
\begin{align*}
   \skp{\bfS_{\mathcal T}(\cdot,\bfD \bfv_h)}{\bfD \bfxi_h} &= \skp{\bff}{\bfxi_h}\qquad
    \forall\,\bfxi_h \in V_{h,\divergence},
\end{align*}
where
\begin{align*}
  V_{h,\divergence} :\!&
  = \set{ \bfw_h \in V_h \,:\, -\skp{\divergence \bfw_h}{\eta_h} =
   0 \qquad\forall\,\eta_h \in Y_h}.
\end{align*}
The coercivity of $\bfS_{\mathcal{T}}$ implies that $\bfD\bfv_h \in
(L^{p_{\mathcal{T}}(\cdot)}(\Omega))^{n\times n}$ with
$\|\bfD\bfv_h\|_{p_{\mathcal T}(\cdot)}\leq c(\bff)$. The next lemma
actually shows that this is equivalent to $\bfD\bfv_h\in
(L^{\px}(\Omega))^{n\times n}$.
\begin{lemma}\label{lems:ppT}
  On the space $\frP_k(\mathcal T)$ the norms $\|\cdot\|_{\px}$
  and $\|\cdot\|_{p_\mathcal T(\cdot)}$ are equivalent. (See also
  Remark~4.7 in~\cite{BrDiSc}.)
\end{lemma}
\begin{proof}
  Let $g_h\in \frP_k(\mathcal{T})$ with $\|g_h\|_{p_\mathcal
    T(\cdot)}\leq 1$ (which is equivalent to $\int_\Omega
  |g_h|^{p_{\mathcal T}(\cdot)}\,dx\leq 1$ by the very definition of
  the Luxemburg norm).  As $g_h$ is a polynomial of order $k$ on $K$
  we have the local estimate (recall $p(x_K) = p^-_K$)
  \begin{align}
    \label{eq:v'1}
    \begin{aligned}
      \|g_h\|_{L^\infty(K)}&\leq \,c(k)\,
      \dashint_{K}|g_h |\,dx
      \leq \,c(k)\,
      \bigg(\dashint_{K}|g_h |^{p(x_K)}\,dx\bigg)^{
        \frac{1}{p(x_K)}}
      \\
      & \leq c(k) \bigg(\frac{1}{h_k^n}\int_{K}|g_h
      |^{p_{\mathcal T}(\cdot)}\,dx\bigg)^{\frac{1}{p(x_K)}} \leq
      c(k)\, h_K^{-\frac{n}{p(x_K)}}. 
    \end{aligned}
  \end{align}
  Thus we can apply Lemma~\ref{lem:pxpy} with $m=\frac{1}{p(x_K)}$,
  $\kappa=0$ and $t=1+|g_h|$ to find
\begin{align*}
  \int_\Omega |g_h|^{\px} \,dx&=\sum_{K\in\mathcal T}\int_K
  |g_h|^{\px} \,dx \leq \sum_{K\in\mathcal T}\int_K (1+|g_h|)^{\px}
  \,dx
  \\
  & \leq c\sum_{K\in\mathcal T}\int_K
  (1+|g_h|)^{p_{\mathcal{T}}} \,dx =c\int_\Omega
  (1+|g_h|)^{p_{\mathcal{T}}} \,dx\leq c.
\end{align*}
On the other hand, if $\|g_h\|_{\px}\leq 1$ there holds
  \begin{align*}
    \|g_h\|_{L^\infty(K)}&\leq
    c \bigg(\dashint_{K}|g_h |^{p_K^{-}}\,dx\bigg)^{
      \frac{1}{p_K^-}}\leq c\,  \bigg(\dashint_{K}1+|g_h |^{\px}\,dx\bigg)^{
      \frac{1}{p_K^-}} \leq c\, h_K^{-\frac{n}{p_K^-}},
  \end{align*}
  and as before $\int_\Omega |g_h|^{p_\mathcal T(\cdot)} \,dx \leq c$.
  \qed
\end{proof}
So we have $\bfD\bfv_h\in (L^{\px}(\Omega))^{n\times n}$ and Korn's
inequality as in~\cite{DiHaHaRu} (Thm. 14.3.21) yields $\bfv_h\in
(W^{1,\px}(\Omega))^{n}$ uniformly in $h$.

In the following we will measure the approximation error in terms of
 the following adapted version of the quasi-norm
\begin{align}
  \norm{\bfF_{\mathcal T}(\cdot,\bfD\bfv) - \bfF_{\mathcal
      T}(\cdot,\bfD\bfv_h)}_2^2 &=\sum_{K\in\mathcal
    T}\int_K\abs{\bfF(x_K,\bfD\bfv) -
    \bfF(x_K,\bfD\bfv_h)}^2\,dx,
  \\
  \text{where }\bfF_{\mathcal T}(x,\bfxi)&:=\sum_{K\in\mathcal T}\chi_K(x)
  \bfF\big(x_K,\bfxi\big).\label{eq:def:F_Tau}
\end{align}
Recall that $\bfF(x,\bfxi)=(\kappa+|\bfxi|)^{\frac{p(x)-2}{2}}\bfxi$.
%
\subsection{Main results}
\label{sec:main_results}
Throughout the paper we will make the following assumptions on our
finite element spaces for approximate velocity and pressure.
\begin{assumption}\label{ass:proj-div}
  We assume that $(\frP_1(\mathcal{T}))^n \subset X_h$ and there
  exists a linear projection operator $\Pidiv\,:\, X \to X_h$ which
  \begin{enumerate}
  \item preserves divergence in the $Y_h^*$-sense, i.e.,
    \begin{align}
      \label{eq:div_preserving}
      \skp{\divergence \bfw}{\eta_h} &= \skp{\divergence \Pidiv
        \bfw}{\eta_h} \qquad \forall\, \bfw \in X,\; \forall\, \eta_h \in
      Y_h\,;
    \end{align}
  \item preserves zero boundary values, i.e. $\Pidiv(V) \subset V_h$;
  \item is locally $W^{1,1}$-stable in the sense that
    \begin{align}
      \label{eq:Pidivcont}
      \dashint_K \abs{\Pidiv\bfw}\,dx &\leq c \dashint_{S_K}\!
      \abs{\bfw}\,dx + c \dashint_{S_K}\! h_K \abs{\nabla \bfw}\,dx
      \quad \forall\, \bfw \in X,\; \forall\, K \in \mathcal{T}.
    \end{align}
  \end{enumerate}
\end{assumption}
\begin{assumption}
  \label{ass:PiY}
  We assume that $Y_h$ contains the constant functions, i.e. $\setR
  \subset Y_h$, and that there exists a linear projection operator
  $\PiY\,:\, Y \to Y_h$ which is locally $L^1$-stable in the sense
  that
  \begin{align}
    \label{eq:PiYstab}
    \dashint_K \abs{\PiY q}\,dx &\leq c\, \dashint_{S_K} \abs{q}\,dx
    \qquad \forall\, q \in Y,\; \forall\, K \in \mathcal{T}.
  \end{align}
\end{assumption}
\begin{remark}{\rm 
  Note that the Cl{\'e}ment and the Scott--Zhang interpolation operators
  satisfy Assumption~\ref{ass:PiY}.}
\end{remark}
\begin{remark}
  \label{rem:Pidivlinpres}
  {\rm It is possible to weaken the requirements on the projection
    operators $\Pidiv$ and $\PiY$. In fact, we can replace the
    requirement $\Pidiv \bfw_h = \bfw_h$ for all $\bfw_h \in X_h$ by
    the requirement $\Pidiv \frq = \boldsymbol \frq$ for all linear
    polynomials (not in the piecewise sense), and the requirement
    $\PiY q_h = q_h$ for all $q_h \in Y_h$ by the requirement $\PiY c
    = c$ for all constants $c$.  }
\end{remark}

Certainly, the existence of $\Pidiv$ depends on the choice of $X_h$
and $Y_h$. Some concrete two-dimensional examples based on the
Scott-Zhang interpolation and a correction of the divergence are
provided in the Appendix of~\cite{BeBeDiRu}.

Let us now state our main results and shortly explain the strategy of
their proofs.  The first result we prove is that the error for the
velocity in the natural distance
is controlled by some best approximation error for the velocity (with
prescribed divergence) and the pressure (cf.~Lemma~\ref{lem:appr1}).
We work directly with a divergence-preserving operator $\Pidiv$
(cf.~Assumption~\ref{ass:proj-div}), which can be used
(cf.~\cite{BF1991}) to derive the inf-sup condition. From the local
$W^{1,1}$-stability of $\Pidiv$, we derive its non-linear, local
counterparts in terms of the natural distance (cf.
Theorem~\ref{thm:app_V}). Thus we can replace the best approximation
error for the velocity (with prescribed divergence) by local averages
of the solution $\bv$ in terms of the natural distance
(cf.~Theorem~\ref{thm:appr2})

Once we have in hand these best approximation estimates we obtain
convergence rates in terms of the mesh size. More precisely we will
prove the following result (see Corollary~\ref{cor:vel} and
Theorem~\ref{thm:appr4}):
\begin{theorem}
  \label{thm:appr_h1}
  Let $\Pidiv$ satisfy Assumption~\ref{ass:proj-div} and $\PiY$
  satisfy Assumption~\ref{ass:PiY}. Let $(\bfv,q)$ and $(\bfv_h,q_h)$
  be solutions of the problems \problemQ{} and \problemQh{},
  respectively. Suppose $p\in C^{0,\alpha}(\overline{\Omega})$ with
  $\alpha\in(0,1]$ and $p^->1$. Furthermore, let $\bfF(\cdot,\bfD
  \bfv) \in (W^{1,2}(\Omega))^{n\times n}$ and also let $q \in
  W^{1,p'(\cdot)}(\Omega)$. Then
  \begin{align}
    \label{eq:appr_h1a}
    \norm{\bfFT(\cdot,\bfD \bfv)-\bfFT(\cdot,\bfD \bfv_h)}_2&\leq c\, \big(h^{\min
      \set{1,\frac {(p^+)'}2}}+h^\alpha\big),
    \\
    \label{eq:appr_h1b}
    \|q-q_h\|_{p'(\cdot)} &\leq
    c\,\Big(h^{\frac{\min\set{((p^+)')^2,4}}{2(p^{-})'}}+h^{\alpha}\Big).
  \end{align}
  Here $c$ depends on $p^-$, $p^+$, $[p]_\alpha$, and $\gamma_0$.
\end{theorem}
\begin{remark}
  It is standard to show $\bfF(\cdot,\bfD \bfv) \in W^{1,2}$ in the
  interior of $\Omega$. The proof follows for instance along the lines
  of~\cite{BFZ,Br} where even more general constitutive relations
  than~\eqref{eq:1.3} were considered.  Note that $\bfF(\cdot,\bfD
  \bfv) \in W^{1,2}$ implies $\bfF(\cdot,\bfD \bfv) \in
  L^{\frac{2n}{n-2}}$ (by Sobolev's Theorem) and $\bfD\bfv\in
  L^{\frac{n}{n-2}p(\cdot)}$. For the space periodic case
  see~\cite{Die2002}. See also~\cite{CG2009} for the problem
   with small data.

  The regularity up to the boundary still seems an open challenging
  problem. The difficulties, even for constant exponents, are due to
  the combination of zero Dirichlet data with the symmetric gradients
  and the pressure. The closest result to $\bfF(\bfD \bfv) \in
  W^{1,2}(\Omega)$ is in~\cite{BeiKapRuz11}, where it is shown for a
  constant $p$ that $\bfF(\bfD \bfv) \in W^{1,s_1}(\Omega)$ and $q \in
  W^{1,s_2}$ for certain $s_1 < 2$ and $s_2 < p'$; see also the
  references therein for other results in this direction.

  In the absence of a pressure and for constants exponents the
  regularity $\bF(\bfD \bfv) \in W^{1,2}(\Omega)$ is shown
  in~\cite{SerShi97}.
\end{remark}
\section{Best Approximation Error for the  Velocity}
\label{sec:velocity}
In this section we prove error estimates for the velocity in terms of
best approximation properties measured in the natural distance.
\subsection{Equation for the  error}
\label{sec:equation-error}
Taking the difference between \problemQ{} and \problemQh{} we get the
following equation for the numerical error
\begin{alignat}{2}
  \label{eq:err_Q}
  \begin{aligned}
    \skp{\bfS(\cdot,\bfD \bfv) - \bfS_{\mathcal T}(\cdot,\bfD \bfv_h)}{\bfD \bfxi_h} -
    \skp{\divergence \bfxi_h}{q-q_h} &= 0 &&\qquad \forall\, \bfxi_h
    \in V_h.
  \end{aligned}
\end{alignat}
We start with a preliminary approximation result which will be improved
later on in Theorem~\ref{thm:appr2}.
\begin{lemma} 
  \label{lem:appr1}
  Let $(\bfv,q)$ and $(\bfv_h,q_h)$ be the solutions of the problems
  \problemQ{} and \problemQh{}, respectively. Suppose $p\in
  C^{0,\alpha}(\overline{\Omega})$ with $\alpha\in(0,1]$ and
  $p^->1$. Then for some $s>1$ (close to 1 for $h$ small) we have the
  following estimate
  \begin{align}
    \label{eq:bestv}
    \begin{aligned}
      \norm{\bfF_{\mathcal T}(\cdot,\bfD \bfv)-\bfF_{\mathcal
          T}(\cdot,\bfD \bfv_h)}_2^2&\leq c\, \inf_{\bfw_h \in
        V_{h,\divergence}} \norm{\bfF_{\mathcal T}(\cdot,\bfD
        \bfv)-\bfF_{\mathcal T}(\cdot,\bfD \bfw_h)}_2^2
      \\
      &\qquad + c\, \inf_{\mu_h \in Y_h} \sum_{K\in\mathcal T}\int_K
      (\phi^K_{\abs{\bfD \bfv}})^\ast (\cdot,\abs{q -
        \mu_h})\,dx
      \\
      &\qquad+c\,h^{2\alpha}\bigg(\int_\Omega(1+|\bfD\bfv|^{p(x)
        s})\,dx\bigg).
    \end{aligned}
  \end{align}
  Here $c$ depends on $p^-$, $p^+$, $[p]_\alpha$, and $\gamma_0$, while
  $(\phi^K_{\abs{\bfD \bfv}})^\ast$ is defined in~\eqref{eq:phiK}.
\end{lemma}
\begin{proof}
  For $\bfw_h \in V_{h,\divergence}$ we have $\bfv_h-\bfw_h \in
  V_{h,\divergence}$. Consequently for all $\mu_h \in Y_h$, we obtain
  with Lemma~\ref{lem:hammer} and Eq.~\eqref{eq:err_Q} that
\begin{align*}
  \norm{\bfF_{\mathcal T}(\cdot,\bfD\bfv) - \bfF_{\mathcal
      T}(\cdot,\bfD\bfv_h)}_2^2 &\leq c\,
  \int_\Omega\big(\bfS_{\mathcal T}(\cdot,\bfD \bfv)-\bfST(\cdot,\bfD
  \bfv_h)\big): \big(\bfD\bfv-\bfD\bfv_h\big)\,dx
  \\
  &= c\, \int_\Omega\big(\bfS(\cdot,\bfD\bfv)-\bfST(\cdot,
  \bfD\bfv_h)\big): \big(\bfD\bfv-\bfD\bfw_h\big)\,dx
  \\
  &\quad +c\,
  \int_\Omega\big(\bfST(\cdot,\bfD\bfv)-\bfS(\cdot,\bfD\bfv)\big):
  \big(\bfD\bfv-\bfD\bfv_h\big)\,dx
  \\
  &\quad -c\, \int_\Omega
  \divergence\big(\bfw_h-\bfv_h\big)\,(q-\mu_h)\,dx
  \\
  &= c\, \int_\Omega\big(\bfST(\cdot,\bfD\bfv)-\bfST(\cdot,
  \bfD\bfv_h)\big): \big(\bfD\bfv-\bfD\bfw_h\big)\,dx
  \\
  &\quad +c\,
  \int_\Omega\big(\bfST(\cdot,\bfD\bfv)-\bfS(\cdot,\bfD\bfv)\big):
  \big(\bfD\bfv-\bfD\bfv_h\big)\,dx
  \\
  &\quad +c\,
  \int_\Omega\big(\bfS(\cdot,\bfD\bfv)-\bfST(\cdot,\bfD\bfv)\big):
  \big(\bfD\bfv-\bfD\bfw_h\big)\,dx
  \\
  &\quad -c\, \int_\Omega
  \divergence\big(\bfw_h-\bfv_h\big)\,(q-\mu_h)\,dx.
\end{align*}
As before this implies
\begin{align*}
  \norm{\bfF_{\mathcal T}(\cdot,\bfD\bfv) - \bfF_{\mathcal
      T}(\cdot,\bfD\bfv_h)}_2^2 &\leq c\,\norm{\bfF_{\mathcal
      T}(\cdot,\bfD\bfv) - \bfF_{\mathcal T}(\cdot,\bfD\bfw_h)}_2^2
  \\
  &\quad +c\,
  \int_\Omega\big(\bfST(\cdot,\bfD\bfv)-\bfS(\cdot,\bfD\bfv)\big):
  \big(\bfD\bfv-\bfD\bfv_h\big)\,dx
  \\
  &\quad +c\,
  \int_\Omega\big(\bfS(\cdot,\bfD\bfv)-\bfST(\cdot,\bfD\bfv)\big):
  \big(\bfD\bfv-\bfD\bfw_h\big)\,dx
  \\
  &\quad -c\, \int_\Omega
  \divergence\big(\bfw_h-\bfv_h\big)\,(q-\mu_h)\,dx
  \\
  &=: (I) + (II) + (III)+(IV).
\end{align*}
We begin with the estimate for~$(II)$.

To estimate the difference between $\bfST$ and $\bfS$ we need the
estimate
\begin{align*}
  \lefteqn{\bigabs{\bfS_{\mathcal T}(x,\bfQ)-\bfS(x,\bfQ)}} \quad &
  \\
  &\leq c\, \abs{p_{\mathcal{T}}(x) - p(x)}\, \abs{\ln(\kappa+
    \abs{\bfQ})} \Big( (\kappa + \abs{\bfQ})^{p_{\mathcal{T}}(x)-2} +
  (\kappa + \abs{\bfQ})^{p(x)-2}\Big) \abs{\bfQ}
  \\
  &\leq c\, h^\alpha \abs{\ln(\kappa+
    \abs{\bfQ})} \Big( (\kappa + \abs{\bfQ})^{p_{\mathcal{T}}(x)-2} +
  (\kappa + \abs{\bfQ})^{p(x)-2}\Big) \abs{\bfQ},
\end{align*}
for all $\bfQ \in \mathbb R^{n\times n}_{sym}$ using also that $p \in
C^{0,\alpha}$. Hence, we get
\begin{align*}
  (II)&:=\int_\Omega\big(\bfS_{\mathcal
    T}(\cdot,\bfD\bfv)-\bfS(\cdot,\bfD\bfv)\big):
  \big(\bfD\bfv-\bfD\bfv_h\big)\,dx
  \\
  &\leq c\, h^\alpha\!\!\int_\Omega \abs{\ln(\kappa\!+\!
    \abs{\bfD\bfv})} (\kappa \!+\!
  \abs{\bfD\bfv})^{p_{\mathcal{T}}(x)-2} \abs{\bfD \bfv}
  \abs{\bfD\bfv\!-\!\bfD\bfv_h}\,dx
  \\
  &\quad + c\, h^\alpha\!\!\int_\Omega \abs{\ln(\kappa\!+\!
    \abs{\bfD\bfv})} (\kappa \!+\!  \abs{\bfD\bfv})^{p(x)-2} \abs{\bfD
    \bfv} \abs{\bfD\bfv\!-\!\bfD\bfv_h}\,dx
  \\
  &=: (II)_1+(II)_2.
\end{align*}
We begin with the estimate for~$(II)_1$ on each~$K \in
\mathcal{T}$. Define the N-function
\begin{align}
\label{eq:phiK}
  \phi^K(t) := \int_0^t (\kappa + s)^{p_K-2}s \,ds.
\end{align}
Using this definition we estimate
\begin{align*}
  (II)_1 &\leq c\, \sum_{K \in \mathcal{T}} \int_K h^\alpha
  \abs{\ln(\kappa\!+\!  \abs{\bfD\bfv})} (\phi^K)'( \abs{\bfD
    \bfv}) \abs{\bfD\bfv\!-\!\bfD\bfv_h}\,dx.
\end{align*}
Using Young's inequality with $\phi^K_{\abs{\bfD \bfv}}:=(\phi^K)_{\abs{\bfD \bfv}}$ on
$\abs{\bfD\bfv\!-\!\bfD\bfv_h}$ and its complementary
function on the rest, we get
\begin{align*}
  (II)_1 &\leq \delta \sum_{K \in \mathcal{T}} \int_K
  (\phi^K)_{\abs{\bfD \bfv}}(\abs{\bfD\bfv\!-\!\bfD\bfv_h})\,dx
  \\
  &+ c_\delta\, \sum_{K \in \mathcal{T}} \int_K \big(
  (\phi^K)_{\abs{\bfD \bfv}} \big)^* \Big( h^\alpha
  \abs{\ln(\kappa\!+\!  \abs{\bfD\bfv})} (\phi^K)'( \abs{\bfD \bfv})
  \Big) \,dx.
\end{align*}
Now we use Lemma~\ref{lem:hammer} for the first line and
Lemma~\ref{lem:shiftedindex} and Lemma~\ref{lem:shifted2}
(with $\lambda=h^\alpha \leq 1$ using $h \leq 1$) for the second line
to find
\begin{align*}
  (II)_1 &\leq \delta c \norm{\bfF_{\mathcal T}(\cdot,\bfD\bfv) -
    \bfF_{\mathcal T}(\cdot,\bfD\bfv_h)}_2^2
  \\
  &+ c_\delta\, \sum_{K \in \mathcal{T}} \int_K
  (1+\abs{\ln(\kappa\!+\!  \abs{\bfD\bfv})})^{\max \set{2,p_K'}} \big(
  (\phi^K)_{\abs{\bfD \bfv}} \big)^* \Big( h^\alpha (\phi^K)'(
  \abs{\bfD \bfv}) \Big) \,dx
  \\
  &\leq \delta c \norm{\bfF_{\mathcal T}(\cdot,\bfD\bfv) -
    \bfF_{\mathcal T}(\cdot,\bfD\bfv_h)}_2^2
  \\
  &+ c_\delta\, \sum_{K \in \mathcal{T}} h^{2\alpha} \int_K
  (1+\abs{\ln(\kappa\!+\!  \abs{\bfD\bfv})})^{\max \set{2,p_K'}}
  (\phi^K)( \abs{\bfD \bfv}) \,dx.
\end{align*}
The term $(II)_2$ is estimate similarly. We get
\begin{align*}
  (II)_2 &\leq \delta c \norm{\bfF_{\mathcal T}(\cdot,\bfD\bfv) -
    \bfF_{\mathcal T}(\cdot,\bfD\bfv_h)}_2^2
  \\
  &\mspace{-30mu}+ c_\delta \sum_{K \in \mathcal{T}} h^{2\alpha}
  \!\!\int_K \big(1\!+\!\abs{\ln(\kappa\!+\!  \abs{\bfD\bfv})} (\kappa
  \!+\!  \abs{\bfD \bfv})^{p(x)-p_\mathcal{T}(x)}\big)^{\max
    \set{2,p_K'}} (\phi^K)( \abs{\bfD \bfv}) \,dx.
\end{align*}
Overall, this yields
\begin{align*}
  (II) &\leq \delta c\,\norm{\bfF_{\mathcal T}(\cdot,\bfD\bfv) -
    \bfF_{\mathcal T}(\cdot,\bfD\bfv_h)}_2^2 + c_\delta c_s
  h^{2\alpha} \int_\Omega \big(1 + \abs{\bfD \bfv}^{p(x)\,s}\big)
  \,dx,
\end{align*}
Here we used $\ln(\kappa+t)\leq c(\kappa)t^\kappa$ for all $t\geq1$ and $\kappa>0$.
For $h$ small we can choose $s$
close to~$1$. 
For $(III)$ the analogous estimate is
\begin{align*}
  (III) &\leq  c\,\norm{\bfF_{\mathcal T}(\cdot,\bfD\bfv) -
    \bfF_{\mathcal T}(\cdot,\bfD\bfw_h)}_2^2 +  c_s
 h^{2\alpha} \int_\Omega \big(1 + \abs{\bfD
    \bfv}^{p(x)\,s}\big) \,dx.
\end{align*}
  Next, we estimate the term $(IV)$ involving $q-\mu_h$. 
  We add and subtract $\bD\bv$, use Young's
  inequality~\eqref{lem:young} for $\phi^K_{\abs{\bfD \bfv}}$, and
  apply Lemma~\ref{lem:hammer} to obtain
 \begin{align*}
   \lefteqn{\big|\skp{\divergence(\bfv_h - \bfw_h)}{q - \mu_h}\big|}
   \hspace{1cm} &
   \\
   &\leq \int_\Omega \big( \abs{\bfD\bfv_h - \bfD\bv} + \abs{\bfD \bv-
     \bfD\bfw_h} \big)\, \abs{q - \mu_h}\,dx
   \\
   &=\sum_{K\in\mathcal T} \int_K \big( \abs{\bfD\bfv_h - \bfD\bv} +
   \abs{\bfD \bv- \bfD\bfw_h} \big)\, \abs{q - \mu_h}\,dx
   \\
   &\leq \epsilon\sum_{K\in\mathcal T}\, \int_\Omega \phi^K_{\abs{\bfD
       \bfv}}(\cdot,\abs{\bfD \bfv_h - \bfD \bfv}) + \phi^K_{\abs{\bfD
       \bfv}}(\cdot,\abs{\bfD \bfw_h - \bfD \bfv})\,dx
   \\
   &\hspace{1cm} + c_\epsilon\sum_{K\in\mathcal T}
   \int_K(\phi^K_{\abs{\bfD \bfv}})^\ast (\cdot,\abs{q - \mu_h})\,dx
   \\
   &\leq \epsilon\,c\,\Big(\,\norm{\bfF_{\mathcal T}(\cdot,\bfD \bfv)
     - \bfF_{\mathcal T}( \cdot,\bfD \bfv_h)}^2_2
   +\norm{\bfF_{\mathcal T}(\cdot,\bfD \bv) - \bfF_{\mathcal
       T}(\cdot,\bfD \bfw_h)}^2_2 \Big)
   \\
   &\hspace{1cm} +c_\epsilon\sum_{K\in\mathcal T}\int_K
   (\phi^K_{\abs{\bfD \bfv}})^\ast (\cdot,\abs{q - \mu_h})\,dx.
  \end{align*}
  Collecting the estimates and choosing $\epsilon>0$ small enough we
  obtain the assertion by noticing that $\bw_h\in V_{h,\divergence}$ and
  $\mu _h\in Y_h$ are arbitrary.
\qed
\end{proof}
\subsection{The divergence-preserving
  interpolation operator}\label{sec:div}
In this section we derive the non-linear estimates for $\PiY$ and the
divergence preserving operator $\Pidiv$. 
\begin{theorem}[{Orlicz-Continuity/Orlicz-Approximability,
  \cite[Section~3]{BrDiSc}}]
  Let $\varphi_a(x,t):=\int_0^t(\kappa+a+s)^{p(x)-2}s\,ds$. Suppose
  $p\in\PPln (\Omega)$ with $p^+<\infty$.
  \begin{itemize}
    \label{lem:PiYstab} \item[a)] Let $\PiY$ satisfy
    Assumption~\ref{ass:PiY}. Then for all $K \in \mathcal{T}$ and $q
    \in L^{\px}(\Omega)$ with
    \begin{align*}
      a+\dashint_Q \abs{q}\,dy \leq \max \set{1, \abs{Q}^{-m}} =
      \abs{Q}^{-m},
    \end{align*}
    we have for every $m\in\mathbb N$ there exists $c_m$ such that

    \begin{align*}
      \int_K \varphi_a \big(\cdot, \abs{\PiY q} \big)\,dx &\leq c_m\,
      \int_{S_K} \varphi_a \big( \abs{q} \big)\,dx+c_m h_K^m.
    \end{align*}
    Moreover, for all $K \in \mathcal{T}$ and $q \in
    W^{1,\px}(\Omega)$ with
    \begin{align*}
      a+\dashint_Q \abs{\nabla q}\,dy \leq \max \set{1, \abs{Q}^{-m}}
      = \abs{Q}^{-m},
    \end{align*}
    we have
    \begin{align*}
      \int_K \varphi_a \big(\cdot, \abs{q- \PiY q} \big)\,dx &\leq
      c_m\, \int_{S_K} \varphi_a \big( h_K \abs{\nabla q}
      \big)\,dx+c_m h_K^m.
    \end{align*}
    \label{thm:ocont}\item[b)]
    Let $\Pidiv$ satisfy Assumption~\ref{ass:proj-div}.  Then $\Pidiv$
    has the local continuity property
    \begin{align*}
      \int_K \varphi_a \big( \cdot,\abs{\nabla \Pidiv \bfw} \big)\,dx
      \leq c\, \int_{S_K} \varphi_a \big(\cdot, \abs{\nabla \bfw}
      \big)\,dx+c_m h_k^m,
    \end{align*}
    for all $K \in \mathcal{T}$ and $\bfw \in
    (W^{1,p(\cdot)}(\Omega))^N$ with
    \begin{align*}
      a+\dashint_Q \abs{\nabla \bfw}\,dy \leq \max \set{1, \abs{Q}^{-m}}
      = \abs{Q}^{-m}.
    \end{align*}
  \end{itemize}
  The constant $c_m$ depends only on $n$, $c_{\log}(p)$, $p^+$, and
  the non-degeneracy constant $\gamma_0$ of the triangulation
  $\mathcal{T}$.
\end{theorem}
\begin{proof}
  a) Due to Assumption~\ref{ass:PiY} the operator $\PiY$ satisfies
  Assumption~1 of~\cite{BrDiSc} both for $r_0=l_0=l=0$ and
  $r_0=l_0=0$, $l=1$. The first choice and~\cite[Corollary
  3.5]{BrDiSc} imply the first assertion, while the second one
  and~\cite[Lemma~3.4]{BrDiSc} yield the second assertion.
  \\
  b) It follows from Assumption~\ref{ass:proj-div} and the usual
  inverse estimates that $\Pidiv$ satisfies Assumption~1
  of~\cite{BrDiSc} with $l=l_0=r_0=1$.  Therefore, the local
  Orlicz-continuity follows from~\cite[Corollary~3.5]{BrDiSc} and the
  local Orlicz-approximability follows from~\cite[Lemma~3.4]{BrDiSc}.
  \qed
\end{proof}

Next, we present the estimates concerning $\Pidiv$ in terms of the
natural distance.
\begin{theorem}
  \label{thm:app_V}
  Let $\Pidiv$ satisfy Assumption~\ref{ass:proj-div}.  Suppose $p\in
  C^{0,\alpha}(\overline{\Omega})$ with $p^->1$ and let~$s>1$. Then we
  have uniformly with respect to $K \in \mathcal{T}$ and to $\bfv \in
  (W^{1,s\px}(\Omega))^n$ 
  \begin{align*}
    \int_K \bigabs{\bfF_{\mathcal T} (\cdot,\bfD \bfv) -
      \bfF_{\mathcal T} (\cdot,\bfD \Pidiv \bfv)}^2 \,dx &\leq c\,
    \int_{S_K}
    \bigabs{\bfF(\cdot,\bfD \bfv) - \mean{\bfF(\cdot,\bD
        \bv)}_{S_K}}^2 \,dx
    \\
    &+c\,h_k^{2\alpha}\int_{S_K} (1+|\bfD\bfv|^{p(x)s})\,dx.
  \end{align*}
  Here $c$ depends on $p^-$, $p^+$, $[p]_\alpha$, $s$, $\gamma_0$, and
  $\norm{\bfD \bfv}_{p(\cdot)}$.
\end{theorem}
\begin{remark}
  In contrast to Lemma 4.7~in~\cite{BrDiSc} we have to deal with
  symmetric gradients instead of full ones.  So we need an appropriate
  version of Korn's inequality (bounding gradients by symmetric
  gradients). A modular version for shifted functions with variable
  exponents is not known in literature (but expected). Instead of this
  we switch to the level of functions and bound an integral depending
  on the function by an integral depending on the symmetric gradient
  (see Theorem~\ref{thm:kornshift}). This is possible if we subtract a
  suitable rigid motion.
\end{remark}
\begin{proof}[of Theorem~\ref{thm:app_V}]
We estimate the best approximation error by the projection error using
$\bfw_h = \Pi_h^{\divergence} \bfv$.
\begin{align*}
  \int_K\abs{\bfF_{\mathcal T}(\cdot,\bfD\bfv) - \bfF_{\mathcal T}(\cdot,\bfD\Pi_h^{\divergence} \bfv)}^2\,dx &\leq
    \int_K\abs{\bfF(\cdot,\bfD\bfv) - \bfF(\cdot,\bfD\Pi_h^{\divergence} \bfv)}^2\,dx
  \\
  &\quad+   \int_K\abs{\bfF_{\mathcal T}(\cdot,\bfD\bfv) -
    \bfF(\cdot,\bfD \bfv)}^2\,dx
  \\
  &\quad +   \int_K\abs{\bfF_{\mathcal T}(\cdot,\bfD\Pi_h^{\divergence} \bfv) - \bfF
    (\cdot,\bfD\Pi_h^{\divergence} \bfv)}^2\,dx
  \\
  &=: [I] + [II] + [III].
\end{align*}
 Note that $\bfv \in (W^{1,\px}(\Omega))^n$ implies
  $\bfF(\cdot,\bfD\bfv) \in (L^2(\Omega))^{n\times n}$.  For
  arbitrary $\bfQ \in \setR^{n \times n}_{sym}$ we have
  \begin{align*}
    \lefteqn{[I]:=\dashint_K \abs{\bfF (\cdot,\bfD \bfv) -
        \bfF(\cdot,\bfD \Pi_h^{\divergence} \bfv)}^2 \,dx}
    \hspace{5mm}
    \\
    &\leq c\, \dashint_K \abs{\bfF(\cdot,\bfD \bfv) -
      \bfF(\cdot,\bfQ)}^2 \,dx + c\, \dashint_K \abs{\bfF (\cdot,\bfD
      \Pi_h^{\divergence} \bfv) - \bfF(\cdot,\bfQ)}^2 \,dx
    \\
    &=: [I]_1 + [I]_2.
  \end{align*}
  Let $\frp \in (\frP_1)^n(S_K)$ be such that $\bfD \frp =
  \bfQ$. Due to $\Pi_h^{\divergence} \frp = \frp$ there holds $\bfQ = \bfD \frp =
  \bfD \Pi_h^{\divergence} \frp$. We estimate by Lemma~\ref{lem:hammer} as follows
  \begin{align}
    [I]_2 &\leq c\, \dashint_K
    \phi_{\abs{\bfQ}+\kappa}\big(\cdot,\abs{\bfD\Pi_h^{\divergence}
      \bfv - \bfQ} \big)\,dx\nonumber
    \\
    &= c\, \dashint_K \phi_{\abs{\bfQ}+\kappa}\big(\cdot,\abs{\bfD \Pi_h^{\divergence}
      (\bfv -\frp) } \big)\,dx.\label{eq:}
  \end{align}
  Now we want to estimate $\abs{\bfD \Pi_h^{\divergence} (\bfv
    -\frp)}$. Since the function $\Pi_h^{\divergence} (\bfv -\frp)$
  belongs to a finite dimensional function space we can apply inverse
  estimates. So we have for every rigid motion $\mathcal R_K$
\begin{align*}
  \|\bfD \Pi_h^{\divergence} (\bfv -\frp) \|_{\infty,K}&\leq
  \,c\,h_K^{-1}\| \Pi_h^{\divergence} (\bfv -\frp) -\mathcal
  R_K\|_{\infty,K}
  \\
  &= \,c\,h_K^{-1}\,\| \Pi_h^{\divergence} (\bfv -\frp -\mathcal
  R_K)\|_{\infty,K}
  \\
  &\leq\,c\,\dashint_K \Big|\frac{ \Pi_h^{\divergence} (\bfv -\frp
    -\mathcal R_K)}{h_K}\Big|\,dy.
\end{align*}
Now applying Theorem~\ref{thm:ocont} and Theorem~\ref{thm:kornshift}
with an appropriate choice of $\mathcal R_K$ yields for
$m_K:=\max\set{n(p_{S_K}^+-2)+2,2}$
\begin{align*}
  \|\bfD \Pi_h^{\divergence} (\bfv -\frp) \|_{\infty,K}
  &\leq\,c\,\dashint_{S_K} \Big|\frac{ (\bfv -\frp -\mathcal
    R_K)}{h_K}\Big|\,dy+c\,h^{m_K}_K
  \\
  &\leq\,c\,\dashint_{S_K} |\bfD(\bfv -\frp)|\,dy+c\,h^{m_K}_K.
\end{align*}
Inserting this in~\eqref{eq:} and using convexity of
$\phi_{\abs{\bfQ}+\kappa}(x,\cdot)$ implies
  \begin{align*}
    [I]_2 &\leq c\,
    \dashint_K\phi_{\abs{\bfQ}+\kappa}\bigg(\cdot,\dashint_{S_K}\abs{\bfD
      (\bfv -\frp) } \big)\,dy +c\,h_K^{m_K}\bigg)\,dx
    \\
    &\leq c\,
    \dashint_K\phi_{\abs{\bfQ}+\kappa}\bigg(\cdot,\dashint_{S_K}\abs{\bfD
      (\bfv -\frp) } \big)\,dy\bigg) +c\,
    \dashint_K\phi_{\abs{\bfQ}+\kappa}(\cdot,h_K^{m_K})\,dx.
  \end{align*}
As a consequence of Theorem~\ref{thm:jensenpxshift} for
  $\bfv -\frp$ , $m=2$ and $a=\abs{\bfQ}$ we gain
  \begin{align*}
    [I]_2 &\leq c\,
    \dashint_{S_K}\phi_{\abs{\bfQ}+\kappa}(\cdot,\abs{\bfD (\bfv
      -\frp) } \big)\,dx+c\,h_K^2 +c\,
    \dashint_K\phi_{\abs{\bfQ}+\kappa}(\cdot,h_K^{m_K})\,dx.
  \end{align*}
  In order to proceed we need a special choice of $\bfQ$.  Following
  the arguments from~\cite{BrDiSc} (Sec.~4) one can show the
  existence of $\bfQ\in\mathbb R^{n\times n}_{sym}$ such that
  \begin{align}
    \label{eq:Q}
    \dashint_{S_K}\bfF(\cdot,\bfQ)\,dx&=\dashint_{S_K}\bfF(\cdot,\bfD\bfv)\,dx,\quad 
    |\bfQ|\leq\,c\,h_K^{-n},
    \\
    \label{eq:neu2}
    \dashint_{S_K} \bigabs{\bfF(\cdot,\bfQ) -
      \langle\bfF(\cdot,\bfQ)\rangle_{S_K}}^2 \,dx 
    &\leq\,c\,h_K^{2\alpha}\,\bigg(\dashint_{S_K}\ln(\kappa+|\bfD\bfv|)^2
    (\kappa+|\bfD\bfv|)^{p(x)}\,dx +1\bigg).
  \end{align}
  Due to~\eqref{eq:Q}, convexity of
  $\phi_{\abs{\bfQ}+\kappa}(x,\cdot)$, and the choice of $m_K$ we have
  \begin{align*}
    [I]_2 &\leq c\, \dashint_{S_K}
    \phi_{\abs{\bfQ}+\kappa}\big(\cdot,\abs{\bfD (\bfv -\frp) }
    \big)\,dx+c\,h_K^2
    +c\,\dashint_{S_K}h^{m_K}_K\phi_{\abs{\bfQ}+\kappa}(\cdot,1)\,dx
    \\
    &\leq c\, \dashint_{S_K}
    \phi_{\abs{\bfQ}+\kappa}\big(\cdot,\abs{\bfD (\bfv -\frp) }
    \big)\,dx+c\,h_K^2 +c\,\dashint_{S_K}h^{m_K}_K
    (1+h_K^{-n(p(\cdot)-2)})\,dx
    \\
    &\leq c\, \dashint_{S_K}
    \phi_{\abs{\bfQ}+\kappa}\big(\cdot,\abs{\bfD \bfv - \bfQ}
    \big)\,dx+c\,h_K^2.
  \end{align*}
  Now, with Lemma~\ref{lem:hammer}
  \begin{align*}
    [I]_2 &\leq c\, \dashint_{S_K} \abs{\bfF(\cdot,\bfD \bfv) -
      \bfF(\cdot,\bfQ)}^2 \,dx+c\,h_K^2.
  \end{align*}
  Since, $\abs{K} \sim \abs{S_K}$ and $K \subset S_K$ we also have
  \begin{align*}
    [I]_1 &\leq c\, \dashint_{S_K} \abs{\bfF(\cdot,\bfD \bfv) -
      \bfF(\cdot,\bfQ)}^2 \,dx.
  \end{align*}
  Overall, we get
  \begin{align*}
  [I] &\leq c\,
    \dashint_{S_K} \bigabs{\bfF(\cdot,\bfD \bfv) -
      \bfF(\cdot,\bfQ)}^2 \,dx+c\,h_K^2,
  \end{align*}
which means we have to estimate the integral on the right-hand-side.
Choosing $\bfQ$
  via~\eqref{eq:Q} and using~\eqref{eq:neu2} we have
  \begin{align*}
    \dashint_{S_K} &\bigabs{\bfF(\cdot,\bfD \bfv) -
      \bfF(\cdot,\bfQ)}^2 \,dx
    \\
    &\leq c \dashint_{S_K} \bigabs{\bfF(\cdot,\bfD \bfv) -
      \langle\bfF(\cdot,\bfD\bfv)\rangle_{S_K}}^2 \,dx+c
    \dashint_{S_K} \bigabs{\bfF(\cdot,\bfQ) -
      \langle\bfF(\cdot,\bfQ)\rangle_{S_K}}^2 \,dx
    \\
    &\leq c \dashint_{S_K} \bigabs{\bfF(\cdot,\bfD \bfv) -
      \langle\bfF(\cdot,\bfD\bfv)\rangle_{S_K}}^2
    \,dx+c\,h_K^{2\alpha} \bigg(\dashint_{S_K}\ln(\kappa+|\bfD\bfv|)^2
    (\kappa+|\bfD\bfv|)^{p(x)}\,dx +1\bigg)
    \\
    &\leq c \dashint_{S_K} \bigabs{\bfF(\cdot,\bfD \bfv) -
      \langle\bfF(\cdot,\bfD\bfv)\rangle_{S_K}}^2
    \,dx+c\,h_K^{2\alpha} \bigg(\dashint_{S_K}
    (1+|\bfD\bfv|)^{s p(x)}\,dx \bigg).
  \end{align*}
The estimate for $[II]$ and $[III]$ are
similar. We have
\begin{align*}
  \lefteqn{\bigabs{\bfF_{\mathcal T}(x,\bfQ)-\bfF(x,\bfQ)}} \quad &
  \\
  &\leq c\, \abs{p_{\mathcal{T}}(x) - p(x)}\, \abs{\ln(\kappa+
    \abs{\bfQ})} \Big( (\kappa +
  \abs{\bfQ})^{\frac{p_{\mathcal{T}}(x)-2}{2}} + (\kappa +
  \abs{\bfQ})^{\frac{p(x)-2}{2}}\Big) \abs{\bfQ}.
\end{align*}
This implies
\begin{align*}
  [II] &\leq c\, h^{2\alpha} \bigg(\int_\Omega (1+ \abs{\bfD
    \bfv})^{s p(x)}\,dx \bigg),
  \\
  [III] &\leq c\, h^{2\alpha} \bigg(\int_\Omega (1+ \abs{\bfD
    \Pi_h^{\divergence} \bfv})^{s p(x)}\,dx \bigg).
\end{align*}
We can use the stability of~$\Pi_h^{\divergence}$, see
Theorem~\ref{thm:ocont} (for $a=0$ and the exponent~$s p(\cdot)$) to
get
\begin{align*}
  [III] &\leq c\, h^{2\alpha} \bigg(\int_\Omega (1+ \abs{\bfD
    \bfv})^{s p(x)}\,dx \bigg). 
\end{align*}
\qed 
\end{proof}
%
\subsection{Error estimate for the velocity}
Collecting the estimates and results of the previous sections we
obtain the most useful error estimate.
\begin{theorem} 
  \label{thm:appr2}
  Let $\Pidiv$ satisfy Assumption~\ref{ass:proj-div}. Let $(\bfv,q)$
  and $(\bfv_h,q_h)$ be solutions of the problems \problemQ{} and
  \problemQh{}, respectively.  Suppose $p\in
  C^{0,\alpha}(\overline{\Omega})$ with $p^->1$ and let $s>1$.  We
  have the following estimate
  \begin{align*}
    \norm{\bfF_{\mathcal T}(\cdot,\bfD \bfv)-\bfF_{\mathcal
        T}(\cdot,\bfD \bfv_h)}_2^2&\leq c\, \sum_{K \in \mathcal{T}}
    \int_{S_K} \bigabs{\bfF(\cdot,\bfD \bfv) - \mean{\bfF(\cdot,\bfD
        \bfv)}_{S_K}}^2 \,dx
    \\
    &\qquad + c\, \inf_{\mu_h \in Y_h} \sum_{K\in\mathcal T}\int_K
    (\phi^K_{\abs{\bfD
        \bfv}})^\ast (\cdot,\abs{q - \mu_h})\,dx
    \\
    &\qquad + c\,h^{2\alpha}\int_{\Omega} (1+|\bfD\bfv|)^{p(x)s}\,dx.
  \end{align*}
  Here $c$ depends on $p^-$, $p^+$, $[p]_\alpha$, $\gamma_0$, and
  $\norm{\bfD \bfv}_{p(\cdot)}$.
\end{theorem}
\begin{proof}
  Since $\Pidiv$ is divergence-preserving
  (see~\eqref{eq:div_preserving}) $\bfv \in V_{\divergence}$ implies
  that $\Pidiv \bfv \in V_{h,\divergence}$. The claim follows from
  Lemma~\ref{lem:appr1} with $\bfw_h := \Pidiv \bfv$ and
  Theorem~\ref{thm:app_V}.
\qed
\end{proof}
\begin{corollary}
\label{cor:vel}
Let the assumptions of Theorem~\ref{thm:appr2} be satisfied.
\\
In addition to all previous hypothesis assume that $\bfF(\cdot,\bfD
\bfv)\in (W^{1,2}(\Omega))^{n\times n}$ 
\\ 
and $q\in
W^{1,p'(\cdot)}$. Then we have
\begin{align*}
  \norm{\bfFT(\cdot,\bfD \bfv)-\bfFT(\cdot,\bfD \bfv_h)}_2&\leq
  c\big(h^{\frac{\min\set{(p^+)',2}}{2}}+h^\alpha\big).
\end{align*}
\end{corollary}
\begin{proof}
  We estimate the three integrals which appear in
  Theorem~\ref{thm:appr2} separately. By Poincar\'{e}'s inequality we
  have
\begin{align*}
  \sum_{K \in \mathcal{T}} \int_{S_K} \bigabs{\bfF(\cdot,\bfD \bfv) -
    \mean{\bfF(\cdot,\bfD \bfv)}_{S_K}}^2 \,dx&\leq \,c\,\sum_{K \in
    \mathcal{T}} \int_{S_K} h_K^2\bigabs{\nabla\bfF(\cdot,\bfD
    \bfv)}^2 \,dx
  \\
  &\leq \,c\,h^2 \int_{\Omega} \bigabs{\nabla\bfF(\cdot,\bfD \bfv)}^2
  \,dx\leq c\,h^2.
\end{align*}
As $\bfF(\cdot,\bfD \bfv)\in W^{1,2}(\Omega)\hookrightarrow
L^{\frac{2n}{n-2}}(\Omega)$ we gain
\begin{align*}
  \int_{\Omega} |\bfD \bfv|^{p(\cdot)s}\,dx<\infty,
\end{align*}
provided $s\leq\frac{n}{n-2}$.  This allows us to bound the third term
by $c\,h^2$. The term involving the pressure requires more effort. We
choose $\mu_h$ by $\mu_h=\Pi_h^Y q$ on $K$ and decompose
\begin{align*}
  \int_\Omega (\phi^K_{\abs{\bfD \bfv}})^\ast &(\cdot,\abs{q -
    \mu_h})\,dx=\sum_{K \in \mathcal{T}} \int_K (\phi^K_{\abs{\bfD
      \bfv}})^\ast (\cdot,\abs{q - \Pi_h^Y q})\,dx
  \\
  &=\sum_{K \in \mathcal{T}^+} \int_K (\phi^K_{\abs{\bfD \bfv}})^\ast
  (\cdot,\abs{q - \Pi_h^Y q})\,dx + \sum_{K \in \mathcal{T}^{-}}
  \int_K (\phi^K_{\abs{\bfD \bfv}})^\ast (\cdot,\abs{q - \Pi_h^Y
    q})\,dx,
\end{align*}
with the abbreviations
\begin{align*}
  \mathcal{T}^+&:=\set{K\in\mathcal T: p_K^{-}\geq 2},
  \\
  \mathcal{T}^{-}&:=\set{K\in\mathcal T: p_K^{-}< 2}.
\end{align*}
For $K \in\mathcal{T}^+$ we have $(\phi^K_{\abs{\bfD \bfv}})^\ast
(\cdot,t)\leq (\phi^K)^\ast (\cdot,t)\leq t^{p_K'(\cdot)}$ such that
\begin{align*}
\int_K (\phi^K_{\abs{\bfD
        \bfv}})^\ast (\cdot,\abs{q - \Pi_h^Y q})\,dx &\leq \int_K
    \abs{q - \Pi_h^Y q}^{p_K'(\cdot)}\,dx. 
  \end{align*}
In the following we will show that
\begin{align}
  \label{eq:do}
\dashint_K \abs{q - \Pi_h^Y q}^{p_K'(\cdot)}\,dx\leq
\,c\,h_K^{(p^+)'}\,\dashint_K \Big( \abs{\nabla
  q}^{p'(\cdot)}+1\Big)\,dx  +c\,h_K^{n+2}. 
\end{align}
We use the identity $q-\PiY q = (q-\mean{q}_{S_K}) -
\PiY(q-\mean{q}_{S_K})$, the triangle inequality together with
$\Delta_2(\phi^*)<\infty$, and the local stability of $\PiY$ from
Lemma~\ref{lem:PiYstab} with $m=n+2$ to conclude that
\begin{align*}
  \dashint_K \abs{q - \Pi_h^Y q}^{p_K'(\cdot)}\,dx &\leq
  \,c\,\dashint_K \abs{q -
    \mean{q}_{S_K}}^{p_K'(\cdot)}\,dx+\,c\dashint_K \abs{\Pi_h^Y(q -
    \mean{q}_{S_K})}^{p_K'(\cdot)}\,dx
  \\
  &=:\{I\} +\{II\}.
\end{align*}
We estimate the first term by
\begin{align*}
  \{I\} &\leq \,c\,\dashint_K \abs{q -
    \mean{q}_{K}}^{p_K'(\cdot)}\,dx+c\,\dashint_K \abs{\mean{q}_K -
    \mean{q}_{S_K}}^{p_K'(\cdot)}\,dx
  \\
  &\leq \,c\,\dashint_K \abs{\nabla q }^{p_K'(\cdot)}\,dx+c\,
  \abs{\mean{q}_K - \mean{q}_{S_K}}^{p_K'(\cdot)}
  \\
  &=:\{I\}_1+\{I\}_2,
\end{align*}
using Poincar\'{e}'s inequality on $L^{p_K'}(K)$. If
$\abs{\mean{q}_K-\mean{q}_{S_K}}\leq h_K^n$ we clearly have
$\{I\}_2\leq c\,h_K^{n+2}$. Otherwise we can use Lemma~\ref{lem:pxpy}
with $m=n$, Theorem~\ref{thm:jensenpxshift} with $a=0$ and
Poincar\'{e}'s inequality from Theorem~\ref{thm:poincareshift} and
gain
\begin{align*}
  \{I\}_2&\leq c\, \abs{\mean{q
      -\mean{q}_{S_K}}_K}^{p'(\cdot)}\leq\,c\,\dashint_{K}|q-\mean{q}_{S_K}|^{p'(\cdot)}\,dx+c\,h_K^{n+2} 
  \\
  &\leq\,c\,\dashint_{S_K}|q-\mean{q}_{S_K}|^{p'(\cdot)}\,dx+c\,h_K^{n+2}\leq
  \,c\,\dashint_{S_K}|h_K\nabla q|^{p'(\cdot)}\,dx+c\,h_K^{n+2}.
\end{align*}
Note that the application of Lemma~\ref{lem:pxpy} was possible as~\ref{lem:PiYstab}
\begin{align*}
  \abs{\mean{q}_K - \mean{q}_{S_K}}&\leq \dashint_K
  |q|\,dx+\dashint_{S_K} |q|\,dx \leq \,c\,\dashint_{S_K} |q|\,dx
  \\
  &\leq |S_K| \|q\|_1\leq c\,h_K^{-n}.
\end{align*}
For $\{II\}$ again we first consider the case $\abs{\Pi_h^Y(q -
  \mean{q}_{S_K})}\leq h_K^n$ in which the estimate is
obvious. Otherwise, we apply Lemma~\ref{lem:pxpy} with $m=n$ as well
as Lemma~\ref{lem:PiYstab} to gain
\begin{align*}
  \{II\}&\leq \,c\,\dashint_K \abs{\Pi_h^Y(q -
    \mean{q}_{S_K})}^{p'(\cdot)}\,dx
  \\
  &\leq \,c\,\,\dashint_K \abs{q -
    \mean{q}_{S_K}}^{p'(\cdot)}\,dx+c\,h_K^{n+2}
  \\
  &\leq \,c\,\,\dashint_{S_K} \abs{q -
    \mean{q}_{S_K}}^{p'(\cdot)}\,dx+c\,h_K^{n+2}
  \\
  &\leq \,c\,\,\dashint_{S_K} \abs{h_K\nabla q
  }^{p'(\cdot)}\,dx+c\,h_K^{n+2}.
\end{align*}
Note that the application of Lemma~\ref{lem:pxpy} is justified since 
\begin{align*}
  \|\Pi_h^Y(q - \mean{q}_{S_K})\|_\infty&\leq \dashint_K
  \abs{\Pi_h^Y(q - \mean{q}_{S_K})}\,dx \leq \,c\,\dashint_{S_K} |q -
  \mean{q}_{S_K}|\,dx
  \\
  &\leq |S_K| \|q\|_1\leq c\,h_K^{-n}.
\end{align*}
Here, we used the inverse estimates on $Y_h$ and
Assumption~\ref{eq:PiYstab}. Finally $(p_K)'\leq p'(x)$ on $K$ yields
the claimed inequality~\eqref{eq:do}. This implies
\begin{align*}
\int_K (\phi^K_{\abs{\bfD
        \bfv}})^\ast (\cdot,\abs{q - \mu_h})\,dx 
    &\leq c\,h^{(p^+)'}\int_K (1+\abs{\nabla q}^{p'(\cdot)})\,dx  +c\,h^{n+2},
\end{align*}
for $K\in\mathcal T^{-}$.  If $K \in\mathcal{T}^{-}$ we estimate
\begin{align*}
  \int_K (\phi^K_{\abs{\bfD \bfv}})^\ast (\cdot,\abs{q - \mu_h})\,dx
  &\leq \int_K \abs{q - \mu_h}^{p_K'(\cdot)}\,dx+\int_K
  (\kappa+|\bfD\bfv|)^{p_K'(\cdot)-2}\abs{q - \mu_h}^2\,dx
  \\
  &\leq \int_K \abs{q - \mu_h}^{p_K'(\cdot)}\,dx+\int_K
  (\kappa+|\bfD\bfv|)^{p_K'(\cdot)-2}\abs{q - \mu_h}^2\,dx.
\end{align*}
The first integral can be estimated via the calculations above,
whereas for the second we gain by Young's inequality and~\eqref{eq:do}
\begin{align*}
  \int_K (\kappa+|\bfD\bfv|)^{p_K'(\cdot)-2}\abs{q -
    \mu_h}^2\,dx&=h_K^2\int_K
  (\kappa+|\bfD\bfv|)^{p_K'(\cdot)-2}\abs{h_K^{-1}(q - \Pi^Y_h)}^2\,dx
  \\
  &\leq h_K^2\bigg(\int_K (\kappa+|\bfD\bfv|)^{p_K(\cdot)}\,dx+\int_K
  \abs{h_K^{-1}(q - \Pi^Y_h)}^{p'(\cdot)}\,dx\bigg)
  \\
  &\leq c\,h^2\bigg(\int_K (1+|\bfD\bfv|^{p(\cdot)s})\,dx+\int_K
  \abs{\nabla q }^{p'(\cdot)}\,dx+h^n\bigg).
\end{align*}
Plugging all together yields
\begin{align*}
  \int_\Omega (\phi^K_{\abs{\bfD \bfv}})^\ast &(\cdot,\abs{q -
    \mu_h})\,dx
  \\
  &\leq c\,h^2\int_\Omega
  (1+|\bfD\bfv|^{p(\cdot)s})\,dx+c\,h^{\min\set{(p^+)',2}}\int_\Omega
  \abs{\nabla q }^{p'(\cdot)}\,dx+c |\mathcal T|h^{2+n}
  \\
  &\leq c\,h^2\int_\Omega
  (1+|\bfD\bfv|^{p(\cdot)s})\,dx+c\,h^{\min\set{(p^+)',2}}\int_\Omega
  \abs{\nabla q }^{p'(\cdot)}\,dx+c\,h^{2}
  \\
  &\leq c\,h^{\min\set{(p^+)',2}}.
\end{align*}
Plugging all estimates together proves the claim.
\qed
\end{proof}
\section{Best Approximation for the pressure}
\label{sec:appr-with-pressure}
We are now discussing best approximation results for the pressure.  As
in the classical Stokes problem we need the discrete inf-sup condition
to recover information on the discrete pressure. We start by extending
this condition to Orlicz spaces.
\subsection{Inf-sup condition on generalized Lebesgue
  spaces} \label{sec:inf-sup} 
The next lemma contains a continuous inf-sup condition. It is
formulated for John domains. Note that all Lipschitz domains and in
particular all polyhedral domains are John domains. We will apply the
following lemmas to simplices~$K$ and their neighborhood~$S_K$, which
have uniform John constants due to the non-degeneracy of the mesh.
For a precise definition of John domains we refer to~\cite{DiHaHaRu}.
\begin{lemma}[\cite{DiHaHaRu}, Thm.~14.3.18]
  \label{lem:ism}
  Let $G \subset \setR^n$ be a John domain and let $p\in \PPln(G)$
  with $1<p^{-}\leq p^+<\infty$. Then, for all $q \in
  L^{p'(\cdot)}_0(G)$ we have
  \begin{align*}
    \norm{q}_{L^{p'(\cdot)}_0(G) } &\le c\, \sup_{\bfxi \in
      W^{1,\px}_0(G)\,:\, \norm{\nabla \bfxi}_{p(\cdot)} \le 1}
    \skp{q}{\divo \bfxi},
  \end{align*}
  where the constants depend only on $p$ and the John constant
  of~$G$.
\end{lemma}
An appropriate discrete version reads as follows.
\begin{lemma}
  \label{lem:ismd}
  Let $G \subset \setR^n$ be a polyhedral domain, let $p \in \PPln(G)$
  with $1<p^- \leq p^+ < \infty$ and let $\Pidiv$ satisfy
  Assumption~\ref{ass:proj-div}. Then for all $q_h \in Q_h$ holds
  \begin{align*}
    \norm{q_h}_{p_\mathcal T'(\cdot)} &\le c\, \sup_{\bfxi_h \in
      V_h\,:\,\norm{\bfxi_h}_{1,p_\mathcal T} \le 1} \skp{q_h}{\divo \bfxi_h},
  \end{align*}
  where the constants depend \footnote{More precisely, on $p$ and the
    John constant of $G$.} only on $p$ and on $G$.
\end{lemma}
\begin{remark}
  Note that the inf-sup condition from Lemma~\ref{lem:ismd} only holds on
  the finite element space $Q_h$. It is not possible to extend it to
  the whole space $L^{p_\mathcal T}(G)$. This is due to the fact that the
  exponent $p_{\mathcal T}$ is not continuous. Log-H\"older
  continuity is a necessary assumption for continuity of singular
  integrals and the maximal function on generalized Lebesgue spaces
  (see~\cite{PiRu}). The inf-sup condition is based on the negative
  norm theorem which follows from the continuity of (the gradient of)
  the \Bogovskii -operator.
\end{remark}
\begin{proof}[of Lemma~\ref{lem:ismd}.]
  We use Lemma~\ref{lem:ism}, Assumption~\ref{ass:proj-div}, and
  Theorem~\ref{thm:ocont} to get
  \begin{align*}
    \norm{q_h}_{ Q_h } &\le c\, \sup_{\norm{\bfxi}_{ V} \le 1}
    \skp{q_h}{\divo \bfxi} = c\, \sup_{\norm{\bfxi}_{V} \le 1}
    \skp{q_h}{\divo \Pidiv \bfxi}
    \\
    &\le c\, \sup_{\norm{\Pidiv \bfxi}_{V_h} \le 1} \skp{q_h}{\divo
      \Pidiv \bfxi} \le c\, \sup_{\norm{\bfxi_h}_{ V_h} \le 1}
    \skp{q_h}{\divo \bfxi_h}.
  \end{align*}
Due to Lemma~\ref{lems:ppT} this is equivalent to the claim.
\qed
\end{proof}
\subsection{Error estimate for the pressure}
We now derive a best approximation result for the numerical error of
the pressure. 
\begin{lemma}
  \label{lem:errpress}
  Let $\Pidiv$ satisfy Assumption~\ref{ass:proj-div}. Let $(\bfv,q)$
  and $(\bfv_h,q_h)$ be solutions of the problems \problemQ{} and
  \problemQh{}, respectively. Suppose $p\in
  C^{0,\alpha}(\overline{\Omega})$ with $p^->1$. Then, we have the
  following estimate
  \begin{align*}
    \|q \!-\!q_h\|_{p_\mathcal T'(\cdot)} & \le c
    \|\bfST(\cdot,\bD\bv) \!-\!\bfST(\cdot,\bD\bv_h)\|_{p_\mathcal
      T'(\cdot)} + c\! \inf _{\mu _h \in Q_h}\|q
    \!-\!\mu_h\|_{p_\mathcal T'(\cdot)} \,+c\,h^\alpha.
  \end{align*}
\end{lemma}
\begin{proof}
  We split the error $q-q_h$ into a best approximation error $q-\mu_h$
  and the remaining part $\mu_h-q_h$, which we will control by means
  of the equation for $q_h$. In particular, for all $\mu_h \in Q_h$ it
  holds
  \begin{align*}
    \|q \!-\!q_h\|_{p_\mathcal T'(\cdot)} &\leq c\, \|q
    \!-\!\mu_h\|_{p_\mathcal T'(\cdot)} + c\, \|\mu_h
    \!-\!q_h\|_{p_\mathcal T'(\cdot)},
  \end{align*}
  by the triangle inequality.  The second term is estimated with the
  help of Lemma~\ref{lem:ismd} as follows
  \begin{align*}
     \|\mu_h \!-\!q_h\|_{p_\mathcal T'(\cdot)} \le  \sup_{\bfxi_h \in
      V_h\,:\,\norm{\bfxi_h}_{1,p_\mathcal T(\cdot)} \le 1} \skp{\mu_h-q_h}{\divo \bfxi_h} .
  \end{align*}
  Let us take a closer look at the term $\skp{\mu_h-q_h}{\divo
    \bfxi_h}$. By using the equation~\eqref{eq:err_Q} for the error,
  we get
  \begin{align*}
    \skp{\mu_h-q_h}{\divo \bfxi_h} & = \skp{\mu_h-q}{\divergence
      \bfxi_h} + \skp{q-q_h}{\divo \bfxi_h}
    \\
    &= \skp{\mu_h-q}{\divo \bfxi_h} + \skp{\bfS(\cdot,\bfD
      \bfv)-\bfS_{\mathcal T}(\cdot,\bfD \bfv_h)}{\bfD \bfxi_h}
    \\
    &= \skp{\mu_h-q}{\divo \bfxi_h} + \skp{\bfST(\cdot,\bfD
      \bfv)-\bfST(\cdot,\bfD \bfv_h)}{\bfD\bfxi_h}
    \\
    &+ \skp{\bfS(\cdot,\bfD \bfv)-\bfS_{\mathcal T}(\cdot,\bfD
      \bfv)}{\bfD \bfxi_h}.
\end{align*}
Applying H\"older's inequality and taking the supremum with respect to
$\bfxi$ yields
\begin{align*} 
  \|\mu_h \!-\!q_h\|_{p_\mathcal T'(\cdot)} &\leq c\, \|q
  \!-\!\mu_h\|_{p_\mathcal T'(\cdot)}+c\, \|\bfST(\cdot,\bD\bv)
  \!-\!\bfST(\cdot,\bD\bv_h)\|_{p_\mathcal T'(\cdot)}
  \\
  &+c \|\bS(\cdot,\bD\bv) \!-\!\bS_{\mathcal
    T}(\cdot,\bD\bv)\|_{p_\mathcal T'(\cdot)}.
  \end{align*}
The last term can be estimated as follows: we have for some $s_1\in(1,s)$
  \begin{align*}
    \int_\Omega &\Big(\frac{\abs{\bfS(\cdot,\bfD \bfv)-\bfS_{\mathcal
          T}(\bfD \bfv)}}{C\,h^\alpha}\Big)^{p_\mathcal T'(\cdot)}\, dx
    \\
    &\qquad\qquad=\sum_{K\in\mathcal T}\int_K
    \Big(\frac{\abs{\bfS(x,\bfD \bfv)-\bfS(x_K,\bfD
        \bfv_h)}}{C\,h^\alpha}\Big)^{p'(x_K)}\, dx
    \\
    &\qquad\qquad\leq \sum_{K\in\mathcal T}\int_K
    \Big(\frac{\log(1+|\bfD\bfv|)(1+|\bfD\bfv_h|)^{p(x_K)-1}}{C_1}\Big)^{p'(x_K)}\,
    dx
    \\
    &\qquad\qquad\leq \sum_{K\in\mathcal T}\int_K
    \Big(\frac{(1+|\bfD\bfv|)^{s_1(p(x_K)-1)}}{C_2}\Big)^{p'(x_K)}\,
    dx.
\end{align*}
Here we took into account $p\in C^{0,\alpha}(\overline{\Omega})$. An
appropriate choice of $C$ (depending on $s_1$ and $s$) implies that, for 
small enough $h>0$,
\begin{align*}
  \int_\Omega &\Big(\frac{\abs{\bfS(\cdot,\bfD \bfv)-\bfS_{\mathcal
        T}(\bfD \bfv)}}{C\,h^\alpha}\Big)^{p'(\cdot)}\, dx\leq
  \frac{1}{C_3}\sum_{K\in\mathcal T}\int_K
  \big(1+|\bfD\bfv|\big)^{p(x)s}\, dx
  \\
  &\qquad\qquad= \frac{1}{C_3}\int_\Omega
  \big(1+|\bfD\bfv|\big)^{p(x)s}\, dx\leq 1.
  \end{align*}
The claim follows, since $\mu_h \in Q_h$ was arbitrary.
\qed
\end{proof}
Unfortunately, the estimate for the error of the pressure ${q-q_h}$
involves the error of the stresses $\bfST(\cdot,\bfD \bfv) -
\bfST(\cdot,\bfD\bfv_h)$. Our error estimates for the velocity in
Theorem~\ref{thm:appr2} are however expressed in terms of
$\bfF_{\mathcal T}(\cdot,\bfD \bfv) - \bfF_{\mathcal T}(\cdot,\bfD
\bfv_h)$.  The following lemma represents the missing link between the
error in terms of $\bfST$ and the error in terms of~$\bfFT$
(with an additional term with respect to the estimate
for fixed $p$).
\begin{lemma}
  \label{lem:errSvsF}
Under the assumptions of Corollary~\ref{cor:vel} It holds
\begin{align}
   \label{eq:errpge2}
       \int_\Omega &\abs{\bfST(\cdot,\bfD \bfv)-\bfST( \cdot,\bfD
         \bfv_h)}^{p'(\cdot)}\, dx \leq
       \,c\,\big(h^{\min\set{\frac{((p^+)')^2}{2},(p^+)'}}+h^{\alpha\min\set{2,(p^+)'}}\big).
\end{align}
\end{lemma}
\begin{proof}
By standard arguments we gain
\begin{align}
  \label{eq:2606}
\begin{aligned}
  \bfST(x,\bfD \bfv)-\bfST(x,\bfD\bfv_h)&=\int_0^1
  D\bfST(x,\bfD\bfv+t(\bfD\bfv_h-\bfD\bfv))\,dt:(\bfD\bfv-\bfD\bfv_h)
  \\
  &\leq \,c\,\int_0^1
  (\kappa+|\bfD\bfv+t(\bfD\bfv_h-\bfD\bfv)|)^{p_{\mathcal
      T}(x)-2}\,dt\quad|\bfD\bfv-\bfD\bfv_h|
  \\
  &\leq \,c\, (\kappa+|\bfD\bfv|+|\bfD\bfv_h-\bfD\bfv|)^{p_{\mathcal
      T}(x)-2}\,|\bfD\bfv-\bfD\bfv_h|.
\end{aligned}
\end{align}
Let us decompose $\mathcal T$ again into $\mathcal T^+$ and $\mathcal
T^-$, where
 \begin{align*}
   \mathcal{T}^+&:=\set{K\in\mathcal T: p_{\mathcal T}> 2},
   \\
   \mathcal{T}^{-}&:=\set{K\in\mathcal T: p_{\mathcal T}\leq 2}.
\end{align*}
It follows for $K\in\mathcal T^{+}$ by Young's inequality for every $\gamma>0$
\begin{align*}
  \int_K &\abs{\bfST(\cdot,\bfD \bfv)-\bfST( \cdot,\bfD
    \bfv_h)}^{p_{\mathcal T}'}\, dx\\&\leq \,c\,\int_K
  (\kappa+|\bfD\bfv|+|\bfD\bfv_h-\bfD\bfv|)^{\frac{p_{\mathcal
        T}-2}{2}p_\mathcal T'}\,|\bfD\bfv-\bfD\bfv_h|^{p_\mathcal
    T'}(1+|\bfD\bfv|+|\bfD\bfv_h|)^{\frac{p_{\mathcal
        T}-2}{2}p_\mathcal T'}\,dx
  \\
  &\leq\,\,c\,\gamma^{-2}\int_K
  (\kappa+|\bfD\bfv|+|\bfD\bfv_h-\bfD\bfv|)^{p_{\mathcal
      T}-2}\,|\bfD\bfv-\bfD\bfv_h|^2\,dx
  \\
  &+\,c\,\gamma^{\frac{2p_{\mathcal T}'}{2-p'_{\mathcal T}}}\int_K
  (1+|\bfD\bfv|+|\bfD\bfv_h|)^{p_{\mathcal T}}\,dx.
\end{align*}
Due to Lemma~\ref{lem:pxpy} and~\ref{eq:v'1} we gain for some $s>1$ as
a consequence of $p\in C^{0,\alpha}(\overline{\Omega})$
\begin{align*}
  \int_K &\abs{\bfST(\cdot,\bfD \bfv)-\bfST( \cdot,\bfD
    \bfv_h)}^{p'_{\mathcal T}}\, dx
  \\
  &\leq\,\,c\,\gamma^{-2}\int_K
  (\kappa+|\bfD\bfv|+|\bfD\bfv_h-\bfD\bfv|)^{p_{\mathcal
      T}(x)-2}\,|\bfD\bfv-\bfD\bfv_h|^2\,dx
  \\
  &+\,c\,\gamma^{\frac{2p_{\mathcal T}'}{2-p'_{\mathcal T}}}\int_K
  (1+|\bfD\bfv|^{\px s}+|\bfD\bfv_h|^{p_\mathcal T(\cdot)})\,dx,
\end{align*}
If $(p^+)'=\min_{K\in \mathcal T^+}\min _K p'<2$ we obtain (in the
other case the following calculations are not necessary because of
$\mathcal T^+=\emptyset$)
\begin{align*}
  \sum_{K\in\mathcal T^+}\int_K &\abs{\bfST(\cdot,\bfD \bfv)-\bfST(
    \cdot,\bfD \bfv_h)}^{p'_{\mathcal T}}\, dx
  \\
  &\leq\,c\gamma^{ \frac{2(p^+)'}{2-(p^+)'}}\,\int_\Omega
  (\kappa+|\bfD\bfv|^{\px s}+|\bfD\bfv_h|^{p_\mathcal
    T(\cdot)})\,dx+\,c\gamma^{-2}\,\int_\Omega \,|\bfF_{\mathcal
    T}(\cdot,\bfD\bfv)-\bfF_{\mathcal T}(\cdot,\bfD\bfv_h)|^{2}\,dx
  \\
  &=:c\Big(\gamma^{ \frac{2(p^+)'}{2-(p^+)'}}A+\gamma^{-2}B\Big).
\end{align*}
We minimize the r.h.s. with respect to $\gamma$ which leads to the optimal choice
\begin{align*}
  \gamma&=\bigg(\frac{2-(p^+)'}{(p^+)'}\bigg)^{\frac{2-(p^+)'}{4}}\bigg(\frac{B}{A
  }\bigg)^{\frac{2-(p^+)'}{4}}\sim
  \bigg(\frac{B}{A}\bigg)^{\frac{2-(p^+)'}{4}}.
\end{align*}  
Note that for $h\ll1$ we can assume that $\gamma\leq1$ as a
consequence of Corollary~\ref{cor:vel}. So we end up with
\begin{align*}
  \sum_{K\in\mathcal T^+}\int_K& \abs{\bfS(\cdot,\bfD \bfv)-\bfS(
    \cdot,\bfD \bfv_h)}^{p'(\cdot)}\, dx
  \\
  \leq&\,c\bigg(\,\int_\Omega (\kappa+|\bfD\bfv|^{\px
    s}+|\bfD\bfv_h|^{p_{\mathcal
      T}(x)})\,dx\bigg)^{\frac{2-(p^+)'}{2}}
  \\
  &\times\bigg(\,\int_\Omega
  \,|\bfFT(\cdot,\bfD\bfv)-\bfFT(\cdot,\bfD\bfv_h)|^{2}\,dx\bigg)^{\frac{(p^+)'}{2}}.
\end{align*}
As a consequence of Corollary~\ref{cor:vel} and $\bfD\bfv_h\in
L^{p_{\mathcal T}}(\Omega)$ uniformly we gain
\begin{align*}
  \sum_{K\in\mathcal T^{+}}\int_K \abs{\bfS(\cdot,\bfD \bfv)-\bfS(
    \cdot,\bfD \bfv_h)}^{p'(\cdot)}\, dx &\leq
  \,c\,\big(h^{\min\set{(p^+)',\frac{((p^+)')^2}{2}}}+h^{(p^+)'\alpha}\big).
\end{align*}
For $K\in\mathcal T^{-}$ we have due to $\bfD\bfv \in L^{\px
  s}(\Omega)$ and $\bfD\bfv_h\in L^{p_{\mathcal T}(\cdot)}(\Omega)$
uniformly in $h$
\begin{align*}
  \int_K &\abs{\bfST(\cdot,\bfD \bfv)-\bfST( \cdot,\bfD
    \bfv_h)}^{p_{\mathcal T}'(\cdot)}\, dx
  \\
  &\leq \,c\,\int_K
  (\kappa+|\bfD\bfv|+|\bfD\bfv_h-\bfD\bfv|)^{p_{\mathcal
      T}(x)-2}\,|\bfD\bfv-\bfD\bfv_h|^2\,dx
  \\
  &\leq\,c\,\int_K
  |\bfFT(\cdot,\bfD\bfv)-\bfFT(\cdot,\bfD\bfv_h)|^{2}\,dx,
\end{align*}
such that
\begin{align*}
  \sum_{K\in\mathcal T^{-}}\int_K &\abs{\bfST(\cdot,\bfD \bfv)-\bfST(
    \cdot,\bfD \bfv_h)}^{p_{\mathcal T}'(\cdot)}\, dx
  \\
  &\leq\,c\,\int_\Omega
  |\bfFT(\cdot,\bfD\bfv)-\bfFT(\cdot,\bfD\bfv_h)|^{2}\,dx
  \\
  &\leq \,c\,\big(h^{\min\set{2,(p^+)'}}+h^{2\alpha}\big).
\end{align*}
Here we used again corollary~\ref{cor:vel}. The claim follows by
combining the estimates for $\mathcal T^+$ and $\mathcal T^-$.
\qed
\end{proof}
Combining Lemma~\ref{lem:errpress} and Lemma~\ref{lem:errSvsF} we get
our desired error estimate for the pressure.
\begin{theorem}
  \label{thm:appr4}
  Let $\Pidiv$ satisfy Assumption~\ref{ass:proj-div}. Let $(\bfv,q)$
  and $(\bfv_h,q_h)$ be solutions of the problems \problemQ{} and
  \problemQh{}, respectively. Assume further that $\bfF(\cdot,\bfD
  \bfv)\in (W^{1,2}(\Omega))^{n\times n}$, $q\in W^{1,p'(\cdot)}$ and
  suppose $p\in C^{0,\alpha}(\overline{\Omega})$ with $p^->1$. Then we
  have for
\begin{equation*}
\alpha_p:=\frac{\alpha}{(p_{-})'}\min\set{2,(p^+)'},\quad
\beta_p:=\frac{1}{(p_{-})'}\min\set{\frac{((p^+)')^2}{2},(p^+)'},
\end{equation*}
the following estimates.
\begin{enumerate}
\item 
    $\|q-q_h\|_{L^{p_\mathcal T'(\cdot)}}\leq \,c\,h^{\min\set{\alpha_p,\beta_p}}$;
  \item $\int_\Omega|q-q_h|^{p_\mathcal
      T'(\cdot)}\,dx\leq\,c\,h^{(p^+)'\min\set{\alpha_p,\beta_p}}$.
\end{enumerate}
\end{theorem}
\begin{proof}
We choose $\mu_h:=\sum_K\chi_K\dashint_K q\,dx$ and apply~\eqref{eq:do} such that
\begin{align*}
\int_K \Big|\frac{q - \Pi_h^Y q}{C h_K}\Big|^{p_K'}\,dx 
      &\leq \,c\,\int_{S_K}\Big( \abs{\nabla q}^{p'(\cdot)}+1\Big)\,dx  +c\,h_K^{n+2}.
\end{align*}
Since $q\in W^{1,p'(\cdot)}$ we obtain
for $C$ large enough
\begin{align*}
  \int_\Omega \Big|\frac{q-\mu_h}{C\,h_K}\Big|^{p_{\mathcal
      T}'(\cdot)}\,dx&=\sum_K \int_K
  \Big|\frac{q-\mu_h}{C\,h_K}\Big|^{p_\mathcal T'(\cdot)}\,dx\leq
  \frac{1}{C'}\sum_K\bigg( \Big(\int_K |\nabla
  q|^{p'(\cdot)}+1\Big)\,dx+h_K^{n+2}\bigg)
  \\
  &\leq \frac{1}{C'}\bigg(\int_\Omega \Big(|\nabla
  q|^{p'(\cdot)}+1\Big)\,dx+1\bigg)=1,
\end{align*}
such that a) follows from Lemma~\ref{lem:errpress} and~\ref{lem:errSvsF}.
\\
In order to show b) we define
$\varkappa(h):=h^{\min\set{\alpha_p,\beta_p}}$ and estimate
\begin{align*}
  \int_\Omega |q-q_h|^{p_{\mathcal T}'(\cdot)}\,dx&\leq
  \,c\,\varkappa(h)^{(p^+)'}\int_\Omega
  \Big|\frac{q-q_h}{C\,\varkappa(h)}\Big|^{p'(\cdot)}\,dx \leq
  c\,\varkappa(h)^{(p^+)'},
\end{align*}
using a).
\qed
\end{proof}
\appendix
\section{\texorpdfstring{Orlicz spaces}{Orlicz spaces}}
\label{sec:Orlicz spaces}
\noindent
The following definitions and results are standard in the theory of
Orlicz spaces and can for example be found in~\cite{RaoR91}.
A continuous, convex function $\rho\,:\, [0,\infty) \to [0,\infty)$
with $\rho(0)=0$, and $\lim_{t \to \infty} \rho(t) = \infty$ is called
a {\em continuous, convex $\phi$-function}.  

We say that $\phi$ satisfies the $\Delta_2$--condition, if there
exists $c > 0$ such that for all $t \geq 0$ holds $\phi(2t) \leq c\,
\phi(t)$. By $\Delta_2(\phi)$ we denote the smallest such
constant. Since $\phi(t) \leq \phi(2t)$ the $\Delta_2$-condition is
equivalent to $\phi(2t) \sim \phi(t)$ uniformly in $t$. For a family
$\phi_\lambda$ of continuous, convex $\phi$-functions we define
$\Delta_2(\set{\phi_\lambda}) := \sup_\lambda \Delta_2(\phi_\lambda)$.
Note that if $\Delta_2(\phi) < \infty$ then $\phi(t) \sim \phi(c\,t)$
uniformly in $t\geq 0$ for any fixed $c>0$.  By $L^\phi$ and
$W^{k,\phi}$, $k\in \setN_0$, we denote the classical Orlicz and
Orlicz-Sobolev spaces, i.e.\ $f \in L^\phi$ iff $\int
\phi(\abs{f})\,dx < \infty$ and $f \in W^{k,\phi}$ iff $ \nabla^j f
\in L^\phi$, $0\le j\le k$.
\\
A $\phi$-function $\rho$ is called a $N$-function iff it is strictly
increasing and convex with
\begin{align*}
  \lim_{t\rightarrow0}\frac{\rho(t)}{t}=
  \lim_{t\rightarrow\infty}\frac{t}{\rho(t)}=0.
\end{align*}
By $\rho^*$ we denote the conjugate N-function of $\rho$, which is
given by $\rho^*(t) = \sup_{s \geq 0} (st - \rho(s))$. Then $\rho^{**}
= \rho$.
\begin{lemma}[Young's inequality]
  \label{lem:young}
  Let $\rho$ be an N-function. Then for all $s,t\geq 0$ we have
  \begin{align*}
    st \leq \rho(s)+\rho^*(t).
  \end{align*}
  If $\Delta_2(\rho,\rho^*)< \infty$, then additionally for all $\delta>0$
  \begin{align*}
    st &\leq \delta\,\rho(s)+c_\delta\,\rho^*(t),
    \\
    st &\leq c_\delta\,\rho(s)+\delta\,\rho^*(t),
    \\
    \rho'(s)t &\leq \delta\,\rho(s)+c_\delta\,\rho(t),
    \\
    \rho'(s)t &\leq \delta\,\rho(t)+c_\delta\,\rho(s),
  \end{align*}
  where $c_\delta= c(\delta, \Delta_2(\set{\rho,\rho^*}))$.
\end{lemma}
\begin{definition}
  \label{ass:phipp}
  Let $\rho$ be an N-function. 
  We say that $\rho$ is {\em elliptic}, if
  $\rho$ is $C^1$ on $[0,\infty)$ and $C^2$ on $(0,\infty)$ and
  assume that 
  \begin{align}
    \label{eq:phipp}
    \rho'(t) &\sim t\,\rho''(t),
  \end{align}
  uniformly in $t > 0$. The constants hidden in $\sim$ are called the
  {\em characteristics of~$\rho$}.
\end{definition}
Note that~\eqref{eq:phipp} is stronger than
$\Delta_2(\rho,\rho^*)<\infty$. In fact, the $\Delta_2$-constants can
be estimated in terms of the characteristics of~$\rho$.

Associated to an elliptic $N$-function $\rho$ we define the tensors
\begin{align*}
  \bfA^\rho(\bfxi)&:=\frac{\rho'(\abs{\bfxi})}{\abs{\bfxi}}\bfxi,\quad
  \bfxi\in\mathbb R^{n\times n}
  \\
  \bfF^\rho(\bfxi)&:=\sqrt{\frac{\rho'(\abs{\bfxi})}{\abs{\bfxi}}}\,\bfxi,\quad
  \bfxi\in\mathbb R^{n\times n}.
\end{align*}

We define the {\em shifted} $N$-function $\rho_a$ for $a\geq 0$ by
\begin{align}
  \label{eq:def_shift}
  \rho_a(t) &:= \int_0^t \frac{\rho'(a+\tau)}{a+\tau} \tau\,d\tau.
\end{align}
The following auxiliary result can be found in~\cite{DieE08,DieR07}.
\begin{lemma}
  \label{lem:shift_sim}
  For all $a,b, t \geq 0$ we have
  \begin{align*}
    \rho_a(t) &\sim 
    \begin{cases}
      \rho''(a) t^2 &\qquad\text{if $t \lesssim a$}
      \\
      \rho(t) &\qquad\text{if $t \gtrsim a$,}
    \end{cases}
    \\
    (\rho_a)_b(t)&\sim\rho_{a+b}(t).
  \end{align*}
\end{lemma}
\begin{lemma}[{\cite[Lemma~2.3]{DieE08}}]
  \label{lem:hammer}
  We have
  \begin{align*}
    \begin{aligned}
      \big({\bfA^\rho}(\bfP) - {\bfA^\rho}(\bfQ)\big) \cdot
      \big(\bfP-\bfQ \big) &\sim \bigabs{ \bfF^\rho(\bfP) -
        \bfF^\rho(\bfQ)}^2
      \\
      &\sim \rho_{\abs{\bfP}}(\abs{\bfP - \bfQ})
      \\
      &\sim \rho''\big( \abs{\bfP} + \abs{\bfQ} \big)\abs{\bfP -
        \bfQ}^2,
    \end{aligned}
  \end{align*}
  uniformly in $\bfP, \bfQ \in \setR^{n \times n}$.  Moreover,
  uniformly in $\bfQ \in \setR^{n \times n}$,
  \begin{align*}
    \bfA^\rho(\bfQ) \cdot \bfQ &\sim \abs{\bfF^\rho(\bfQ)}^2\sim
    \rho(\abs{\bfQ})
    \\
    \abs{{\bfA^\rho}(\bfP) -
      {\bfA^\rho}(\bfQ)}&\sim\big(\rho_{\abs{\bfP}}\big)'(\abs{\bfP -
      \bfQ}).
  \end{align*}
  The constants depend only on the characteristics of $\rho$.
\end{lemma} 
\begin{lemma}[Change of Shift]
  \label{lem:shift_ch}
  Let $\rho$ be an elliptic N-function. Then for each $\delta>0$ there
  exists $C_\delta \geq 1$ (only depending on~$\delta$ and the
  characteristics of~$\rho$) such that
  \begin{align*}
    \rho_{\abs{\bfa}}(t)&\leq C_\delta\, \rho_{\abs{\bfb}}(t)
    +\delta\, \rho_{\abs{\bfa}}(\abs{\bfa - \bfb}),
    \\
    (\rho_{\abs{\bfa}})^*(t)&\leq C_\delta\, (\rho_{\abs{\bfb}})^*(t)
    +\delta\, \rho_{\abs{\bfa}}(\abs{\bfa - \bfb}),
  \end{align*}
  for all $\bfa,\bfb\in\setR^n$ and $t\geq0$.
\end{lemma}
The case $\bfa=0$ or $\bfb=0$ implies the following corollary.
\begin{corollary}[Removal of Shift]
  \label{cor:shift_ch}
  Let $\rho$ be an elliptic N-function. Then for each $\delta>0$ there
  exists $C_\delta \geq 1$ (only depending on~$\delta$ and the
  characteristics of~$\rho$) such that
  \begin{align*}
    \rho_{\abs{\bfa}}(t)&\leq C_\delta\, \rho(t) +\delta\,
    \rho(\abs{\bfa}),
    \\
    \rho(t)&\leq C_\delta\, \rho_{\abs{\bfa}}(t) +\delta\,
    \rho(\abs{\bfa}),
  \end{align*}
  for all $\bfa\in\setR^n$ and $t\geq0$.
\end{corollary}
\begin{lemma}
  \label{lem:shifted2}
  Let $\rho$ be an elliptic N-function. Then $(\rho_a)^*(t) \sim
  (\rho^*)_{\rho'(a)}(t)$ uniformly in $a,t \geq 0$. Moreover, for all
  $\lambda \in [0,1]$ we have
  \begin{align*}
    \rho_a(\lambda a) &\sim \lambda^2 \rho(a) \sim
    (\rho_a)^*(\lambda \rho'(a)).
  \end{align*}
\end{lemma}
\begin{lemma}
  \label{lem:shiftedindex}
  Let $\rho(t) := \int_0^t (\kappa+s)^{q-2} s\,ds$ with $q \in
  (1,\infty)$ and $t\geq 0$. Then 
  \begin{align*}
    \rho_a(\lambda t) &\leq c\, \max\set{\lambda^q, \lambda^2}
    \rho(t),
    \\
    (\rho_a)^*(\lambda t) &\leq c\, \max\set{\lambda^{q'}, \lambda^2}
    \rho(t),
  \end{align*}
  uniformly in $a,\lambda \geq 0$.
\end{lemma}
\begin{remark}
  Let $p\in\PP(\Omega)$ with $p^{-}>1$ and $p^+<\infty$.  The
  results above extend to the function $\phi(x,t)=\int_0^t
  (\kappa+s)^{p(x)-2}s\,ds$ uniformly in $x\in\Omega$, where the
  constants only depend on $p^-$ and $p^+$.
\end{remark}
\def\polhk#1{\setbox0=\hbox{#1}{\ooalign{\hidewidth
  \lower1.5ex\hbox{`}\hidewidth\crcr\unhbox0}}}
  \def\ocirc#1{\ifmmode\setbox0=\hbox{$#1$}\dimen0=\ht0 \advance\dimen0
  by1pt\rlap{\hbox to\wd0{\hss\raise\dimen0
  \hbox{\hskip.2em$\scriptscriptstyle\circ$}\hss}}#1\else {\accent"17 #1}\fi}
  \def\ocirc#1{\ifmmode\setbox0=\hbox{$#1$}\dimen0=\ht0 \advance\dimen0
  by1pt\rlap{\hbox to\wd0{\hss\raise\dimen0
  \hbox{\hskip.2em$\scriptscriptstyle\circ$}\hss}}#1\else {\accent"17 #1}\fi}
  \def\ocirc#1{\ifmmode\setbox0=\hbox{$#1$}\dimen0=\ht0 \advance\dimen0
  by1pt\rlap{\hbox to\wd0{\hss\raise\dimen0
  \hbox{\hskip.2em$\scriptscriptstyle\circ$}\hss}}#1\else {\accent"17 #1}\fi}
  \def\ocirc#1{\ifmmode\setbox0=\hbox{$#1$}\dimen0=\ht0 \advance\dimen0
  by1pt\rlap{\hbox to\wd0{\hss\raise\dimen0
  \hbox{\hskip.2em$\scriptscriptstyle\circ$}\hss}}#1\else {\accent"17 #1}\fi}
  \def\cprime{$'$} \def\ocirc#1{\ifmmode\setbox0=\hbox{$#1$}\dimen0=\ht0
  \advance\dimen0 by1pt\rlap{\hbox to\wd0{\hss\raise\dimen0
  \hbox{\hskip.2em$\scriptscriptstyle\circ$}\hss}}#1\else {\accent"17 #1}\fi}
  \def\polhk#1{\setbox0=\hbox{#1}{\ooalign{\hidewidth
  \lower1.5ex\hbox{`}\hidewidth\crcr\unhbox0}}} \def\cprime{$'$}
  \def\cprime{$'$} \def\cprime{$'$} \def\cprime{$'$}
\providecommand{\bysame}{\leavevmode\hbox to3em{\hrulefill}\thinspace}
\providecommand{\MR}{\relax\ifhmode\unskip\space\fi MR }
\providecommand{\MRhref}[2]{%
  \href{http://www.ams.org/mathscinet-getitem?mr=#1}{#2}
}
\providecommand{\href}[2]{#2}

\end{document}